\documentclass[11pt,twoside,reqno]{amsart}

\usepackage{amssymb, latexsym}
\usepackage[italian,english]{babel}
\usepackage{amsmath,amsfonts,amsthm}
\usepackage{paralist}
\usepackage[top=2.50cm, bottom=2.50cm, left=2.50cm, right=2.50cm]{geometry}

\usepackage{color}

\numberwithin{equation}{section}

\usepackage{xspace}
\usepackage[bookmarksnumbered,colorlinks]{hyperref}
\usepackage{graphics}
\def\bb#1\eb{\textcolor{blue}
{#1}} %
\def\br#1\er{\textcolor{red}
{#1}} %
\def\bv#1\ev{\textcolor{green}
{#1}} %
\def\bc#1\ec{\textcolor{cyan}
{#1}} %

\usepackage{graphics}

\def\Xint#1{\mathchoice
  {\XXint\displaystyle\textstyle{#1}}%
  {\XXint\textstyle\scriptstyle{#1}}%
  {\XXint\scriptstyle\scriptscriptstyle{#1}}%
  {\XXint\scriptscriptstyle\scriptscriptstyle{#1}}%
  \!\int}
\def\XXint#1#2#3{{\setbox0=\hbox{$#1{#2#3}{\int}$}
  \vcenter{\hbox{$#2#3$}}\kern-.5\wd0}}
\def\-int{\Xint -}

\newcommand{\e}{\varepsilon}
\newcommand{\R}{\mathbb{R}}
\newcommand{\N}{\mathbb{N}}

\DeclareMathOperator{\B}{\mathcal{B}}

\newcommand{\h}{H^{s}(\R^{N})}
\newcommand{\p}{2^{*}_{s}}

\newtheorem{prop}{Proposition}[section]
\newtheorem{lem}{Lemma}[section]
\newtheorem{thm}{Theorem}[section]

\newtheorem{remark}{Remark}[section]

\begin{document}
\title[Fractional Choquard equation]{Multiplicity and concentration results for a fractional Choquard equation via penalization method}
\author[V. Ambrosio]{Vincenzo Ambrosio}
\address{Dipartimento di Scienze Pure e Applicate (DiSPeA),
Universit\`a degli Studi di Urbino `Carlo Bo'
Piazza della Repubblica, 13
61029 Urbino (Pesaro e Urbino, Italy)}
\email{vincenzo.ambrosio@uniurb.it}


\date{}

\begin{abstract}
This paper is devoted to the study of the following fractional Choquard equation
$$
\e^{2s}(-\Delta)^{s} u + V(x)u = \e^{\mu-N}\left(\frac{1}{|x|^{\mu}}*F(u)\right)f(u) \mbox{ in } \R^{N},
$$
where $\varepsilon>0$ is a parameter, $s\in (0, 1)$, $N>2s$, $(-\Delta)^{s}$ is the fractional Laplacian, $V$ is a positive continuous potential with local minimum, $0<\mu<2s$, and $f$ is a superlinear continuous function with subcritical growth.
By using the penalization method and the Ljusternik-Schnirelmann theory, we investigate the multiplicity and concentration of positive solutions for the above problem.
\end{abstract}

\maketitle

\section{Introduction}

\noindent
In this paper we deal with the following nonlinear fractional Choquard equation
\begin{equation}\label{P}
\e^{2s}(-\Delta)^{s} u + V(x)u = \e^{\mu-N}\left(\frac{1}{|x|^{\mu}}*F(u)\right)f(u) \mbox{ in } \R^{N},
\end{equation}
where $\e>0$ is a parameter, $s\in (0,1)$, $N>2 s$ and $0<\mu<2s$.
The potential $V: \R^{N}\rightarrow \R$ is a continuous function verifying the following hypotheses:
\begin{compactenum}[$(V_1)$]
\item $\inf_{x\in \R^{N}} V(x)=V_{0}>0$;
\item there exists a bounded open set $\Lambda\subset \R^{N}$ such that
$$
V_{0}<\min_{x\in \partial \Lambda} V(x). 
$$
\end{compactenum}
Concerning the nonlinearity $f: \R\rightarrow \R$, we assume that $f$ is a continuous function such that $f(t)=0$ for $t<0$, and satisfies the following conditions: 
\begin{compactenum}[($f_1$)]
\item $\displaystyle{\lim_{t\rightarrow 0} \frac{f(t)}{t}=0}$;
\item there exists $q\in (2, \frac{2^{*}_{s}}{2}(2-\frac{\mu}{N}))$, where $2^{*}_{s}=\frac{2N}{N-2s}$, such that $\displaystyle{\lim_{t\rightarrow \infty} \frac{f(t)}{t^{q-1}}=0}$;
\item $f$ verifies the following Ambrosetti-Rabinowitz type condition \cite{AR}:
\begin{equation*}
0< 4 F(t) = 4 \int_{0}^{t} f(\tau) \, d\tau \leq 2f(t) \, t \, \mbox{ for all } t> 0;
\end{equation*}
\item The map $\displaystyle{t \mapsto \frac{f(t)}{t}}$ is increasing for every $t>0$.
\end{compactenum}
The nonlocal operator $(-\Delta)^{s}$ is the fractional Laplacian which may be defined for any $u:\R^{N}\rightarrow \R$ sufficiently smooth by
$$
(-\Delta)^{s}u(x)= -\frac{C(N,s)}{2} \int_{\R^{N}} \frac{u(x+y)+u(x-y)-2u(x)}{|y|^{N+2s}} \,dy  \quad (x\in \R^{N}),
$$
where $C(N,s)$ is a suitable normalization constant; see for instance \cite{DPV, MBRS}.

We recall that the problem (\ref{P}) is motivated by the search of standing wave solutions for the following fractional Schr\"odinger equation
\begin{equation*}
\imath \frac{\partial \psi}{\partial t}=(-\Delta)^{s} \psi+V(x)\psi -\left(\frac{1}{|x|^{\mu}}*|\Psi|^{q}\right)|\Psi|^{q-2}\Psi \quad (t, x)\in \R\times \R^{N},
\end{equation*}
which naturally models many physical problems, such as phase transition, conservation laws, especially in fractional quantum mechanics. For physical motivations we refer to \cite{A2, A3, DDPW, FFV, FQT, FL, HZ, Laskin2, Secchi1}. \\
When $s=1$, $V(x)\equiv 1$, $\e=1$ and $F(u)=\frac{|u|^{2}}{2}$, \eqref{P} boils down to the Choquard-Pekar equation
\begin{equation}\label{CE}
-\Delta u + u = \left(\frac{1}{|x|^{\mu}}*|u|^{2}\right)u \mbox{ in } \R^{N}
\end{equation}
introduced by Pekar \cite{Pek} to describe the quantum mechanics of a polaron.  
Subsequently, Choquard used \eqref{CE} to describe an electron trapped in its own hole as approximation to Hartree-Fock Theory of one component plasma; see \cite{LL, Pen}.\\
The early existence and symmetry results are due to Lieb \cite{Lieb} and Lions \cite{Lions}.
Later, Ma and Zhao \cite{MZ} obtained some qualitative properties of positive solutions considering powers like $|u|^{q}$. Moroz and Van Shaftingen \cite{MVS1} investigated regularity, radial symmetry and asymptotic behavior at infinity of positive solutions for a generalized Choquard equation. Alves and Yang \cite{AY3} studied multiplicity and concentration of positive solutions for a Choquard equation. Further results on Choquard equations can be found in \cite{Ack, AY1, MVS3, Secchi0, YZZ, WW}. 
\noindent

In the case $s\in (0, 1)$, only few recent papers considered fractional Choquard equations like \eqref{P}. 
In \cite{DSS} d'Avenia et al. considered the following fractional Choquard equation 
$$
(-\Delta)^{s}u+u=\left(\frac{1}{|x|^{\mu}}*|u|^{p}\right)|u|^{p-2}u \mbox{ in } \R^{N},
$$
obtaining regularity,  existence and non existence, symmetry and decay properties of solutions.
Frank and Lenzman \cite{FL} established uniqueness of nonnegative ground states for the $L^{2}$ critical boson star equation
$$
(-\Delta)^{\frac{1}{2}}u+u=\left(\frac{1}{|x|^{\mu}}*|u|^{2}\right)u \mbox{ in } \R^{3},
$$
by using variational methods and the extension technique \cite{CS}.
Coti Zelati and Nolasco \cite{CZN} obtained existence of ground state solutions for a pseudo-relativistic Hartree-equation via critical point theory.
Shen et al. \cite{SGY} investigated the existence of ground state solutions for a fractional Choquard equation involving a nonlinearity satisfying Berestycki-Lions type assumptions. Chen and Liu \cite{CL} studied an autonomous fractional Choquard equation via Nehari manifold and concentration-compactness arguments.
Belchior et al. \cite{BBMP} dealt with existence, regularity and polynomial decay for a fractional Choquard equation involving the fractional $p$-Laplacian. \\
Motivated by the above papers, in this work we focus our attention on the multiplicity and the concentration of positive solutions of (\ref{P}), involving a potential and a continuous nonlinearity satisfying the assumptions $(V_1)$-$(V_2)$ and $(f_1)$-$(f_4)$ respectively. In particular, we are interested in relating the number of positive solutions of (\ref{P}) with the topology of the set $M=\{x\in \Lambda: V(x)=V_{0}\}$. In order to state precisely our result, we recall that if $Y$ is a given closed set of a topological space $X$, we denote by $cat_{X}(Y)$ the Ljusternik-Schnirelmann category of $Y$ in $X$, that is the least number of closed and contractible sets in $X$ which cover $Y$.\\
The main result of this paper is the following:
\begin{thm}\label{thmf}
Suppose that $V$ verifies $(V_1)$-$(V_2)$,  $0<\mu<2s$ and $f$ satisfies $(f_1)$-$(f_4)$ with $2<q<\frac{2(N-\mu)}{N-2s}$. Then, for any $\delta>0$ such that $M_{\delta}=\{x\in \R^{N}: dist(x, M)\leq \delta\}\subset \Lambda$, there exists $\e_{\delta}>0$ such that, for any $\e\in (0, \e_{\delta})$, the problem  \eqref{P} has at least $cat_{M_{\delta}}(M)$ positive solutions. Moreover, if $u_{\e}$ denotes one of these positive solutions and $x_{\e}\in \R^{N}$ its global maximum, then 
$$
\lim_{\e\rightarrow 0} V(x_{\e})=V_{0}.
$$
\end{thm}

\noindent
Firstly, we note that the restriction on the exponent $q$ is justified by the Hardy-Littlewood-Sobolev inequality (see Theorem \ref{HLS} in Section $2$).
Indeed, if $F(u)=|u|^{q}$, then the term
$$
\int_{\R^{N}} \left(\frac{1}{|x|^{\mu}}*F(u)\right) F(u)\, dx
$$
is well-defined if $F(u)\in L^{t}(\R^{N})$ for $t>1$ such that $\frac{2}{t}+\frac{\mu}{N}=2$. Hence, recalling that $H^{s}(\R^{N})$ is continuously embedded into $L^{r}(\R^{N})$ for any $r\in [2, 2^{*}_{s}]$, we need to require that $tq\in [2, \p]$, which leads to assume that 
$$
2-\frac{\mu}{N}\leq q\leq \frac{\p}{2}\left(2-\frac{\mu}{N}\right).
$$
Now, we give a sketch of the proof of Theorem \ref{thmf}. Inspired by \cite{AY3}, we adapt the del Pino-Felmer penalization technique \cite{DPF} considering an auxiliary problem. It consists in making a suitable modification on the nonlinearity $f$, solving a modified problem and then check that, for $\e$ sufficiently small, the solutions of the modified problem are indeed solutions of the original one. Differently from the case $s=1$, in our setting a more accurate investigation is needed due to the presence of  two non-local terms. 
Moreover, the nonlinearity appearing in \eqref{P} is only continuous (while $f\in C^{1}$ in \cite{AY3}), so to overcome the non-differentiability of the associated Nehari manifold, we will use some abstract critical point results due to Szulkin and Weth \cite{SW}. 
Concerning the multiplicity result for the modified problem, we resemble some ideas due to Benci and Cerami in \cite{BC}, based on the comparison between the category of some sublevel sets of the modified functional and the category of the set $M$. Finally, in order to prove that the solutions $u_{\e}$ of the modified problem are solutions of the problem \eqref{P}, we adapt a Moser iteration argument \cite{Moser} to establish $L^{\infty}$-estimates, and after showed that the convolution term remains bounded, we exploit some useful properties of the Bessel kernel \cite{AM, FQT}  to obtain the desired result.
To our knowledge, this is the first result in which the concentration and the multiplicity of solutions to \eqref{P} are considered by using penalization argument and Ljusternik-Schnirelmann category theory.
\noindent

The plan of the paper is the following: In Section $2$ we introduce the functional setting and the modified problem. The Section $3$ is devoted to the existence of positive solutions to the autonomous problem associated to \eqref{P}. In Section $4$, we obtain a multiplicity result using Ljusternik-Schnirelmann theory. Finally, exploiting a Moser iteration scheme, we are able to prove that for $\e$ small enough, the solutions of the modified problem are indeed solutions of \eqref{P}.

\section{Functional setting}
\noindent
For any $s\in (0,1)$, we denote by $\mathcal{D}^{s, 2}(\R^{N})$ the completion of $C^{\infty}_{0}(\R^{N})$ with respect to
$$
[u]^{2}=\iint_{\R^{2N}} \frac{|u(x)-u(y)|^{2}}{|x-y|^{N+2s}} \, dx \, dy =\|(-\Delta)^{\frac{s}{2}} u\|^{2}_{L^{2}(\R^{N})},
$$
that is
$$
\mathcal{D}^{s, 2}(\R^{N})=\left\{u\in L^{2^{*}_{s}}(\R^{N}): [u]_{H^{s}(\R^{N})}<\infty\right\}.
$$
Let us introduce the fractional Sobolev space
$$
H^{s}(\R^{N})= \left\{u\in L^{2}(\R^{N}) : \frac{|u(x)-u(y)|}{|x-y|^{\frac{N+2s}{2}}} \in L^{2}(\R^{2N}) \right \}
$$
endowed with the natural norm 
$$
\|u\| = \sqrt{[u]^{2} + \|u\|_{L^{2}(\R^{N})}^{2}}.
$$

\noindent
We collect the following useful results.
\begin{thm}\cite{DPV}\label{Sembedding}
Let $s\in (0,1)$ and $N>2s$. Then there exists a sharp constant $S_{*}=S(N, s)>0$
such that for any $u\in H^{s}(\R^{N})$
\begin{equation}\label{FSI}
\|u\|^{2}_{L^{2^{*}_{s}}(\R^{N})} \leq S_{*}^{-1} [u]^{2}. 
\end{equation}
Moreover $H^{s}(\R^{N})$ is continuously embedded in $L^{q}(\R^{N})$ for any $q\in [2, 2^{*}_{s}]$ and compactly in $L^{q}_{loc}(\R^{N})$ for any $q\in [2, 2^{*}_{s})$. 
\end{thm}

\begin{lem}\cite{FQT}\label{lions lemma}
Let $N>2s$. If $(u_{n})$ is a bounded sequence in $H^{s}(\R^{N})$ and if
$$
\lim_{n \rightarrow \infty} \sup_{y\in \R^{N}} \int_{B_{R}(y)} |u_{n}|^{2} dx=0
$$
where $R>0$,
then $u_{n}\rightarrow 0$ in $L^{t}(\R^{N})$ for all $t\in (2, 2^{*}_{s})$.
\end{lem}

\begin{thm}\label{HLS}\cite{LL}
Let $r, t>1$ and $0<\mu<N$ such that $\frac{1}{r}+\frac{\mu}{N}+\frac{1}{t}=2$. Let $f\in L^{r}(\R^{N})$ and $h\in L^{t}(\R^{N})$. Then there exists a sharp constant $C(r, N, \mu, t)>0$ independent of $f$ and $h$ such that 
$$
\int_{\R^{N}}\int_{\R^{N}} \frac{f(x)h(y)}{|x-y|^{\mu}}\, dx dy\leq C(r, N, \mu, t)\|f\|_{L^{r}(\R^{N})}\|h\|_{L^{t}(\R^{N})}.
$$
\end{thm}

\noindent
For any $\e>0$, we denote by $H^{s}_{\e}$ the completion of $C^{\infty}_{0}(\R^{N})$ with respect to the norm
$$
\|u\|^{2}_{\e}=\iint_{\R^{2N}} \frac{|u(x)-u(y)|^{2}}{|x-y|^{N+2s}}\, dx dy+\int_{\R^{N}} V(\e x) u^{2}(x)\, dx.
$$
It is clear that $H^{s}_{\e}$ is a Hilbert space with respect to the inner product 
$$
(u, v)_{\e}=\iint_{\R^{2N}} \frac{(u(x)-u(y))}{|x-y|^{N+2s}}(v(x)-v(y))\, dx dy+\int_{\R^{N}} V(\e x) u v\, dx.
$$

\noindent
By using the change of variable $u(x)\mapsto u(\e x)$ we can see that the problem (\ref{P}) is equivalent to the following one
\begin{equation}\label{R}
(-\Delta)^{s} u + V(\e x)u =  \left(\frac{1}{|x|^{\mu}}*F(u)\right)f(u)  \mbox{ in } \R^{N}.  
\end{equation}
Fix $\ell>2$ and $a>0$ such that $\frac{f(a)}{a}=\frac{V_{0}}{\ell}$, and we introduce the functions
$$
\tilde{f}(t):=
\begin{cases}
f(t)& \text{ if $t \leq a$} \\
\frac{V_{0}}{\ell} t   & \text{ if $t >a$},
\end{cases}
$$ 
and
$$
g(x, t)=\chi_{\Lambda}(x)f(t)+(1-\chi_{\Lambda}(x))\tilde{f}(t),
$$
where $\chi_{\Lambda}$ is the characteristic function on $\Lambda$, and  we write $G(x, t)=\int_{0}^{t} g(x, \tau)\, d\tau$.\\
Let us note that from the assumptions $(f_1)$-$(f_4)$, $g$ satisfies the following properties:
\begin{compactenum}[($g_1$)]
\item $\displaystyle{\lim_{t\rightarrow 0} \frac{g(x, t)}{t}=0}$ uniformly in $x\in \R^{N}$;
\item $\displaystyle{\lim_{t\rightarrow \infty} \frac{g(x, t)}{t^{q-1}}=0}$ uniformly in $x\in \R^{N}$;
\item $0< 4 G(x, t)\leq 2g(x, t)t$ for any $x\in \Lambda$ and $t>0$, and \\
 $0\leq 2 G(x, t)\leq g(x, t)t\leq \frac{V_{0}}{\ell}t^{2}$ for any $x\in \R^{N}\setminus \Lambda$ and $t>0$,
\item $t\mapsto g(x, t)$ and $t\mapsto \frac{G(x, t)}{t}$ are increasing for all $x\in \R^{N}$ and $t>0$.
\end{compactenum}
Thus we consider the following auxiliary problem 
\begin{equation}\label{Pe}
(-\Delta)^{s} u + V(\e x)u =  \left(\frac{1}{|x|^{\mu}}*G(\e x, u)\right)g(\e x, u) \mbox{ in } \R^{N}
\end{equation}
and we note that if $u$ is a solution of (\ref{Pe}) such that 
\begin{equation}\label{ue}
u(x)<a \mbox{ for all } x\in \R^{N}\setminus \Lambda_{\e},
\end{equation}
where $\Lambda_{\e}:=\{x\in \R^{N}: \e x\in \Lambda\}$, then $u$ solves (\ref{R}), in view of the definition of $g$.\\
It is clear that, weak solutions to (\ref{Pe}) are critical points of the $C^{1}$-functional $J_{\e}: H^{s}_{\e}\rightarrow \R$ defined by
$$
J_{\e}(u)=\frac{1}{2}\|u\|^{2}_{\e}-\Sigma_{\e}(u)
$$
where 
$$
\Sigma_{\e}(u)=\frac{1}{2}\int_{\R^{N}} \left(\frac{1}{|x|^{\mu}}*G(\e x, u)\right)G(\e x, u)\, dx.
$$
We begin proving that $J_{\e}$ satisfies the assumptions of the mountain pass theorem \cite{AR}. 
\begin{lem}\label{MPG}
$J_{\e}$ has a mountain pass geometry, that is
\begin{compactenum}[(i)]
\item there exist $\alpha, \rho>0$ such that $J_{\e}(u)\geq \alpha$ for any $u\in H^{s}_{\e}$ such that $\|u\|_{\e}=\rho$;
\item there exists $e\in H^{s}_{\e}$ with $\|e\|_{\e}>\rho$ such that $J_{\e}(e)<0$.
\end{compactenum}
\end{lem}
\begin{proof}
From $(g_1)$ and $(g_2)$, it follows that that for any $\eta>0$ there exists $C_{\eta}>0$ such that
\begin{equation}\label{g-estimate}
|g(\e x, t)|\leq \eta |t|+C_{\eta} |t|^{q-1}.
\end{equation}
By using Theorem \ref{HLS} and \eqref{g-estimate}, we get
\begin{align}\label{a1}
\left|\int_{\R^{N}} \left(\frac{1}{|x|^{\mu}}*G(\e x, u)\right)G(\e x, u)\, dx\right|\leq C\|G(\e x, u)\|_{L^{t}(\R^{N})} \|G(\e x, u)\|_{L^{t}(\R^{N})} \leq C\left(\int_{\R^{N}} (|u|^{2}+|u|^{q}\, dx)^{t}\right)^{\frac{2}{t}},
\end{align}
where $\frac{1}{t}=\frac{1}{2}(2-\frac{\mu}{N})$. Since $2<q<\frac{2^{*}_{s}}{2}(2-\frac{\mu}{N})$, we can see that $t q\in (2, 2^{*}_{s})$, and from Theorem \ref{Sembedding},  we have
\begin{align}\label{a2}
\left(\int_{\R^{N}} (|u|^{2}+|u|^{q}\, dx)^{t}\right)^{\frac{2}{t}}\leq C(\|u\|^{2}_{\e}+\|u\|^{q}_{\e})^{2}.
\end{align}
Taking into account \eqref{a1} and \eqref{a2} we can deduce that
\begin{align*}
\left|\int_{\R^{N}} \left(\frac{1}{|x|^{\mu}}*G(\e x, u)\right)G(\e x, u)\, dx\right|\leq C(\|u\|^{2}_{\e}+\|u\|^{q}_{\e})^{2}\leq C(\|u\|^{4}_{\e}+\|u\|^{2q}_{\e}).
\end{align*}
As a consequence
$$
J(u)\geq \frac{1}{2}\|u\|^{2}_{\e}-C(\|u\|^{4}_{\e}+\|u\|^{2q}_{\e}),
$$
and being $q>2$ we can see that $(i)$ holds.
Fix a positive function $u_{0}\in H^{s}(\R^{N})\setminus\{0\}$ with $supp(u_{0})\subset \Lambda_{\e}$, and we set
$$
h(t)=\Sigma_{\e}\left(\frac{t u_{0}}{\|u_{0}\|_{\e}}\right) \mbox{ for } t>0.
$$
Since $G(\e x, u_{0})=F(u_{0})$ and by using $(f_3)$, we deduce that
\begin{align}\label{a3}
h'(t)&=\Sigma_{\e}'\left(\frac{t u_{0}}{\|u_{0}\|_{\e}}\right) \frac{u_{0}}{\|u_{0}\|_{\e}} \nonumber \\
&=\int_{\R^{N}} \left(\frac{1}{|x|^{\mu}}*F\left(\frac{t u_{0}}{\|u_{0}\|_{\e}}\right)  \right) f\left(\frac{t u_{0}}{\|u_{0}\|_{\e}}\right)\frac{u_{0}}{\|u_{0}\|_{\e}}\, dx\nonumber \\
&=\frac{4}{t}\int_{\R^{N}} \frac{1}{2} \left(\frac{1}{|x|^{\mu}}*F\left(\frac{t u_{0}}{\|u_{0}\|_{\e}}\right)  \right) \frac{1}{2}  f\left(\frac{t u_{0}}{\|u_{0}\|_{\e}}\right)\frac{t u_{0}}{\|u_{0}\|_{\e}}\, dx \nonumber\\
&>\frac{4}{t} h(t).
\end{align}
Integrating \eqref{a3} on $[1, t\|u_{0}\|_{\e}]$ with $t>\frac{1}{\|u_{0}\|_{\e}}$, we find
$$
h(t\|u_{0}\|_{\e})\geq h(1)(t\|u_{0}\|_{\e})^{4}
$$
which gives
$$
\Sigma_{\e}(t u_{0})\geq \Sigma_{\e}\left(\frac{u_{0}}{\|u_{0}\|_{\e}}\right) \|u_{0}\|_{\e}^{4}t^{4}.
$$
Therefore, we have
$$
J_{\e}(t u_{0})= \frac{t^{2}}{2}\|u_{0}\|_{\e}^{2}-\Sigma_{\e}(t u_{0})\leq C_{1} t^{2}-C_{2}t^{4} \mbox{ for } t>\frac{1}{\|u_{0}\|_{\e}}.
$$
Taking $e=t u_{0}$ with $t$ sufficiently large, we can see that $(ii)$ holds.
\end{proof}

\noindent
Since $supp(u_{0})\subset \Lambda_{\e}$, there exists $\kappa>0$ independent of $\e, l, a$ such that 
$$
c_{\e}= \inf_{u\in H^{s}_{\e}\setminus \{0\}}\max_{t\geq 0} J_{\e}(tu) <\kappa.
$$
Now, let us define
$$
\B=\{u\in H^{s}(\R^{N}): \|u\|^{2}_{\e}\leq 4(\kappa+1)\}
$$
and we set
$$
\tilde{K}_{\e}(u)(x)=\frac{1}{|x|^{\mu}}*G(\e x, u).
$$
We prove the following useful lemma.
\begin{lem}\label{lemK}
Assume that $(f_1)$-$(f_3)$ hold and $2<q<\frac{2(N-\mu)}{N-2s}$. Then there exists $\ell_{0}>0$ such that 
$$
\frac{\sup_{u\in \B} \|\tilde{K}_{\e}(u)(x)\|_{L^{\infty}(\R^{N})}}{\ell_{0}}<\frac{1}{2} \mbox{ for any } \e>0.
$$
\end{lem}
\begin{proof}
Let us prove that there exists $C_{0}>0$ such that 
\begin{equation}\label{a6}
\sup_{u\in \B} \|\tilde{K}_{\e}(u)(x)\|_{L^{\infty}(\R^{N})}\leq C_{0}.
\end{equation}
We observe that
\begin{equation}\label{a5}
|G(\e x, u)|\leq C(|u|^{2}+|u|^{q}) \mbox{ for all } \e>0.
\end{equation}
By using \eqref{a5}, we can see that
\begin{align}\label{a7}
|\tilde{K}_{\e}(u)(x)|&=\left| \int_{\R^{N}} \frac{G(\e x, u)}{|x-y|^{\mu}} \,dy\right| \nonumber\\
&\leq \left| \int_{|x-y|\leq 1} \frac{G(\e x, u)}{|x-y|^{\mu}} \,dy\right|+\left| \int_{|x-y|>1} \frac{G(\e x, u)}{|x-y|^{\mu}} \,dy\right| \nonumber\\
&\leq C \int_{|x-y|\leq 1} \frac{|u(y)|^{2}+|u(y)|^{q}}{|x-y|^{\mu}}\, dy+C \int_{\R^{N}} (|u|^{2}+|u|^{q})\, dy \nonumber\\
&\leq C \int_{|x-y|\leq 1} \frac{|u(y)|^{2}+|u(y)|^{q}}{|x-y|^{\mu}}\, dy+C
\end{align}
where in the last line we used Theorem \ref{Sembedding} and $\|u\|^{2}_{\e}\leq 4(\kappa+1)$.
Now, we take 
$$
t\in \left(\frac{N}{N-\mu}, \frac{N}{N-2s}\right] \mbox{ and } r\in \left(\frac{N}{N-\mu}, \frac{2N}{q(N-2s)}\right].
$$
By applying H\"older inequality and by using Theorem \ref{Sembedding} and $\|u\|^{2}_{\e}\leq 4(\kappa+1)$, we can see that
\begin{align}\label{a8}
\int_{|x-y|\leq 1} \frac{|u(y)|^{2}}{|x-y|^{\mu}}\, dy&\leq \left(\int_{|x-y|\leq 1} |u|^{2t}\, dy  \right)^{\frac{1}{t}} \left(\int_{|x-y|\leq 1} \frac{1}{|x-y|^{\frac{t\mu}{t-1}}}\, dy \right)^{\frac{t-1}{t}}\nonumber \\
&\leq C_{*}(4(\kappa+1))^{2} \left(\int_{\rho\leq 1} \rho^{N-1-\frac{t \mu}{t-1}}\, d\rho  \right)^{\frac{t-1}{t}}<\infty.
\end{align}
because of $N-1-\frac{t \mu}{t-1}>-1$.
Similarly, we get
\begin{align}\label{a9}
\int_{|x-y|\leq 1} \frac{|u(y)|^{q}}{|x-y|^{\mu}}\, dy&\leq \left(\int_{|x-y|\leq 1} |u|^{rq}\, dy  \right)^{\frac{1}{r}} \left(\int_{|x-y|\leq 1} \frac{1}{|x-y|^{\frac{r\mu}{r-1}}}\, dy \right)^{\frac{r-1}{r}}\nonumber \\
&\leq C_{*}(4(\kappa+1))^{q} \left(\int_{\rho\leq 1} \rho^{N-1-\frac{r \mu}{r-1}}\, d\rho  \right)^{\frac{r-1}{r}}<\infty
\end{align}
in view of $N-1-\frac{r \mu}{r-1}>-1$.
Putting together \eqref{a8} and \eqref{a9} we can see that
$$
\int_{|x-y|\leq 1} \frac{|u(y)|^{2}+|u(y)|^{q}}{|x-y|^{\mu}}\, dy\leq C \mbox{ for all } x\in \R^{N}
$$
which in view of \eqref{a7} yields \eqref{a6}.
Then there exists $\ell_{0}>0$ such that
$$
\frac{\sup_{u\in \B} \|\tilde{K}_{\e}(u)(x)\|_{L^{\infty}(\R^{N})}}{\ell_{0}}\leq \frac{C_{0}}{\ell_{0}}< \frac{1}{2}.
$$
\end{proof}

\noindent
Then, we take $a>0$ the unique number such that 
$$
\frac{f(a)}{a}=\frac{V_{0}}{\ell_{0}}
$$
and we consider the penalized problem \eqref{Pe} with these choices.
\noindent
Now, let us introduce the Nehari manifold associated to (\ref{Pe}), that is
\begin{equation*}
\mathcal{N}_{\e}:= \{u\in H^{s}_{\e} \setminus \{0\} : \langle J_{\e}'(u), u \rangle =0\}.
\end{equation*}
By using Theorem \ref{HLS} and $(g_1)$-$(g_2)$, we can see that for all $u\in \mathcal{N}_{\e}$
$$
\|u\|^{2}_{\e}\leq C(\|u\|_{\e}^{4}+\|u\|_{\e}^{2q}),
$$
so there exists $r>0$ such that 
\begin{equation}\label{BBN}
\|u\|_{\e}\geq r \mbox{ for all } u\in \mathcal{N}_{\e}, \e>0.
\end{equation}
Let us denote by $\mathbb{S}_{\e}$ the unitary sphere in $H^{s}_{\e}$.
Since $f$ is only continuous, the next two results will play a fundamental role
to overcome the non-differentiability of $\mathcal{N}_{\e}$.
\begin{lem}\label{lemz1}
Suppose that $V$ satisfies $(V_1)$-$(V_2)$ and $f$ verifies $(f_1)$-$(f_4)$. \\
Then, the following facts hold true:
\begin{compactenum}
\item[$(a)$] 
For any $u\in H^{s}_{\e}\setminus\{0\}$, let $h_{u}: \R_{+} \rightarrow \R$ be defined by $h_{u}(t):= J_{\e}(t u)$. Then, there is a unique $t_{u}>0$ such that $h_{u}'(t)>0$ in $(0, t_{u})$ and $h_{u}'(t)<0$ in $(t_{u}, \infty)$.
\item[$(b)$] There is $\tau>0$, independent on $u$, such that $t_{u}\geq \tau$ for every $u\in \mathbb{S}_{\e}$. Moreover, for each compact set $\mathcal{W}\subset \mathbb{S}_{\e}$, there is $C_{\mathcal{W}}>0$ such that $t_{u}\leq C_{\mathcal{W}}$ for every $u\in \mathcal{W}$.
\item[$(c)$] The map $\hat{m}_{\e}: H^{s}_{\e}\setminus\{0\}\rightarrow \mathcal{N}_{\e}$ given by $\hat{m}_{\e}(u):=t_{u}u$ is continuous and $m_{\e}:= \hat{m}|_{\mathbb{S}_{\e}}$ is a homeomorphism between $\mathbb{S}_{\e}$ and $\mathcal{N}_{\e}$. Moreover, $m^{-1}_{\e}(u)= \frac{u}{\|u\|_{\e}}$. 
\end{compactenum}
\end{lem}
\begin{proof}
$(a)$ From the proof of Lemma \ref{MPG} we can see that $h_{u}(0)=0$, $h_{u}(t)>0$ for $t$ small and $h_{u}(t)<0$ for $t$ large.
Then, by the continuity of $h_{u}$, it is easy to see that there exists $t_{u}>0$ such that $\max_{t\geq 0} h_{u}(t)=h_{u}(t_{u})$, $t_{u}u\in \mathcal{N}_{\e}$ and $h_{u}'(t_{u})=0$. \\
Now, we note that 
\begin{align*}
t u\in \mathcal{N}_{\e} \Longleftrightarrow  \|u\|^{2}_{\e}=\int_{\R^{N}} \left(\frac{1}{|x|^{\mu}}*\frac{G(\e x, t u)}{t} \right) g(\e x, tu) u\, dx,
\end{align*}
so, by using $(g_4)$, we get the uniqueness of a such $t_{u}$. 

$(b)$ Let $u\in \mathbb{S}_{\e}$. Recalling that $h'_{u}(t_{u})=0$, and using $(g_1)$-$(g_2)$, Theorem \ref{HLS} (see estimates in Lemma \ref{lemMPG}), and Theorem \ref{Sembedding}, we get for any $\xi>0$ small 
$$
t_{u}^{2}=\int_{\R^{N}} \tilde{K}_{\e}(t_{u}u) g(\e x, t_{u} u)t_{u}u \, dx\leq  \xi C_{1} t_{u}^{4}+C_{2} C_{\xi} t_{u}^{2q}.
$$
Being $q>2$, there exists $\tau>0$ independent of $u$, such that $t_{u}\geq \tau$.
Now, by using $(g_3)$, we can observe that 
\begin{align}\label{absurd}
J_{\e}(v)&=J_{\e}(v)-\frac{1}{4}\langle J'_{\e}(v), v\rangle \nonumber \\
&=\frac{1}{4}\|v\|_{\e}^{2}-\frac{1}{4}\int_{\R^{N}} \left(\frac{1}{|x|^{\mu}}*G(\e x, u) \right)[2G(\e x, u)-g(\e x, u)u]\, dx\nonumber\\
&\geq \frac{1}{4}\|v\|_{\e}^{2} \quad \mbox{ for any } v\in \mathcal{N}_{\e}.
\end{align}
Hence, if $\mathcal{W}\subset \mathbb{S}_{\e}$ is a compact set, and $(u_{n})\subset \mathcal{W}$ is such that $t_{u_{n}}\rightarrow \infty$, it follows that $u_{n}\rightarrow u$ in $H^{s}_{\e}$, and $J_{\e}(t_{u_{n}}u_{n})\rightarrow -\infty$. Taking $v_{n}=t_{u_{n}}u_{n}\in \mathcal{N}_{\e}$ in (\ref{absurd}), we can see that 
$$
0<\frac{1}{4}\leq \frac{J_{\e}(t_{u_{n}}u_{n})}{t^{2}_{u_{n}}}\leq 0 \mbox{ as } n\rightarrow \infty
$$
which gives a contradiction. \\
$(c)$ Since $(a)$ and $(b)$ hold, we can apply Proposition $8$ in \cite{SW} to deduce the thesis.
\end{proof}

\begin{remark}
From the estimates in $(b)$, we can deduce that $J_{\e}$ is coercive on $\mathcal{N}_{\e}$, because  for all $u\in \mathcal{N}_{\e}$
$$
J_{\e}(u)\geq \frac{1}{4}\|u\|^{2}_{\e}\rightarrow \infty \mbox{ as } \|u\|_{\e}\rightarrow \infty.
$$
Taking into account the above estimate and \eqref{BBN}, we can also see that $J_{\e}|_{\mathcal{N}_{\e}}$ is bounded below by some positive constant.
\end{remark}

\noindent
Let us define the maps $\hat{\psi}_{\e}: H^{s}_{\e}\setminus\{0\} \rightarrow \R$ by $\hat{\psi}_{\e}(u):= J_{\e}(\hat{m}_{\e}(u))$, and $\psi_{\e}:=\hat{\psi}|_{\mathbb{S}_{\e}}$. 
The next result is a consequence of Lemma \ref{lemz1}. For more details, see Proposition $9$ and Corollary $10$ in \cite{SW}.
\begin{prop}\label{propz2}
Suppose that $V$ satisfies $(V_1)$-$(V_2)$ and $f$ verifies $(f_1)$-$(f_4)$. Then, one has:
\begin{compactenum}
\item[$(a)$] $\hat{\psi}_{\e}\in C^{1}(H^{s}_{\e}\setminus\{0\}, \R)$ and
\begin{equation*}
\langle \hat{\psi}_{\e}'(u), v \rangle=\frac{\|\hat{m}_{\e}(u)\|_{\e}}{\|u\|_{\e}} \langle J_{\e}'(\hat{m}_{\e}(u)), v \rangle \,,
\end{equation*}
for every $u\in H^{s}_{\e}\setminus\{0\}$ and $v\in H^{s}_{\e}$;
\item[$(b)$] $\psi_{\e} \in C^{1}(\mathbb{S}_{\e}, \R)$ and $\langle \psi'_{\e}(u), v \rangle = \|m_{\e}(u)\|_{\e} \langle J_{\e}'(m_{\e}(u)), v \rangle$, for every $v\in T_{u}\mathbb{S}_{\e}$.
\item[$(c)$] If $(u_{n})$ is a $(PS)_{d}$ sequence for $\psi_{\e}$, then $(m_{\e}(u_{n}))$ is a $(PS)_{d}$ sequence for $J_{\e}$. Moreover, if $(u_{n})\subset \mathcal{N}_{\e}$ is a bounded $(PS)_{d}$ sequence for $J_{\e}$, then $(m_{\e}^{-1} (u_{n}))$ is a $(PS)_{d}$ sequence for the functional $\psi_{\e}$;
\item[$(d)$] $u$ is a critical point of $\psi_{\e}$ if and only if $m_{\e}(u)$ is a nontrivial critical point for $J_{\e}$. Moreover, the corresponding critical values coincide and
\begin{equation*}
\inf_{u\in\mathbb{S}_{\e}} \psi_{\e}(u) = \inf_{u\in\mathcal{N}_{\e}} J_{\e}(u).
\end{equation*}
\end{compactenum}
\end{prop}

\begin{remark}
As in \cite{SW}, we have the following characterization of the infimum of $J_{\e}$ on $\mathcal{N}_{\e}$:
$$
c_{\e}= \inf_{u\in\mathcal{N}_{\e}} J_{\e}(u)=\inf_{u\in H^{s}_{\e}\setminus\{0\}} \max_{t>0} J_{\e}(t u)=\inf_{u\in \mathbb{S}_{\e}} \max_{t>0} J_{\e}(t u).
$$
\end{remark}

\noindent
In the next result we show that $J_{\e}$ verifies a local compactness condition.
\begin{lem}\label{PSc}
 $J_{\e}$ satisfies the $(PS)_{c}$ condition for all $c\in [c_{\e}, \kappa]$.
\end{lem}
\begin{proof}
Let $(u_{n})$ be a Palais-Smale sequence at the level $c$, that is $J_{\e}(u_{n})\rightarrow c$ and $J_{\e}'(u_{n})\rightarrow 0$.
We divide the proof in two main steps.\\
{\bf Step $1$}: For any $\eta>0$ there exists $R=R_{\eta}>0$ such that 
\begin{equation}\label{DF}
\limsup_{n\rightarrow \infty}\int_{\R^{N}\setminus B_{R}} |(-\Delta)^{\frac{s}{2}}u_{n}|^{2}+V(\e x)u_{n}^{2}\, dx<\eta.
\end{equation}
By using $(g_3)$, we can see that
\begin{align*}
J_{\e}(u_{n})-\frac{1}{4}\langle J'_{\e}(u_{n}), u_{n}\rangle&=\left(\frac{1}{2}-\frac{1}{4}\right)\|u_{n}\|^{2}_{H^{s}_{\e}}+\frac{1}{4} \int_{\R^{N}} \left(\frac{1}{|x|^{\mu}}*G(\e x, u_{n})\right) g(\e x, u_{n}) u_{n}\, dx \\
&-\frac{1}{2}  \int_{\R^{N}} \left(\frac{1}{|x|^{\mu}}*G(\e x, u_{n})\right) G(\e x, u_{n}) \, dx\\
&\geq \frac{1}{4} \|u_{n}\|^{2}_{\e},
\end{align*}
so there exists $n_{0}\in \N$ such that 
$$
\|u_{n}\|^{2}_{\e}\leq 4(\kappa+1) \mbox{ for all } n\geq n_{0}.
$$
Therefore, we may assume that $u_{n}\rightharpoonup u$ in $H^{s}(\R^{N})$ and $u_{n}\rightarrow u$ in $L^{r}_{loc}(\R^{N})$ for any $r\in [2, 2^{*}_{s})$. 
Moreover, by Lemma \ref{lemK}, we can deduce that
\begin{equation}\label{Kbound}
\frac{\sup_{n\geq n_{0}}\|\tilde{K}_{\e}(u_{n})(x)\|_{L^{\infty}(\R^{N})}}{\ell_{0}}\leq \frac{1}{2}.
\end{equation}
Fix $R>0$ and let $\psi_{R}\in C^{\infty}(\R^{N})$ be a function such that $\psi_{R}=0$ in $B_{R/2}$, $\psi_{R}=1$ in $B_{R}^{c}$, $\psi_{R}\in [0, 1]$ and $|\nabla \eta_{R}|\leq C/R$.
Since $(u_{n}\psi_{R})$ is bounded we can see that
\begin{align}
\int_{\R^{N}} (-\Delta)^{\frac{s}{2}} u_{n} (-\Delta)^{\frac{s}{2}} (u_{n}\psi_{R})+V(\e x)\psi_{R}u_{n}^{2} \, dx &=\langle J_{\e}'(u_{n}), u_{n}\psi_{R}\rangle+\int_{\R^{N}} \left(\frac{1}{|x|^{\mu}}*G(\e x, u_{n})\right) g(\e x, u_{n})u_{n}\psi_{R} \nonumber\\
&=o_{n}(1)+\int_{\R^{N}} \left(\frac{1}{|x|^{\mu}}*G(\e x, u_{n})\right) g(\e x, u_{n})u_{n}\psi_{R}\, dx
\end{align}
For $n\geq n_{0}$ and $\e>0$ fixed, take $R>0$ big enough such that $\Lambda_{\e}\subset B_{R/2}$.
Then, by using $(g_3)$ with $\ell_{0}$ as in Lemma \ref{lemK}, we deduce that
\begin{align*}
\int_{\R^{N}\setminus B_{R/2}} |(-\Delta)^{\frac{s}{2}}u_{n}|^{2}+V(\e x)u_{n}^{2}\, dx&\leq \int_{\R^{N}\setminus B_{R/2}} \left(\frac{1}{|x|^{\mu}}*G(\e x, u_{n})\right) g(\e x, u_{n})u_{n}\, dx+o_{n}(1) \\
&-\iint_{\R^{2N}} \frac{(u_{n}(x)-u_{n}(u))(\psi_{R}(x)-\psi_{R}(y))}{|x-y|^{N+2s}} u_{n}(y)\, dx dy\\
&\leq \int_{\R^{N}\setminus B_{R/2}}  \frac{\sup_{n\geq n_{0}}\|\tilde{K}_{\e}(u_{n})(x)\|_{L^{\infty}(\R^{N})}}{\ell_{0}} V_{0}u_{n}^{2}\, dx+o_{n}(1) \\
&-\iint_{\R^{2N}} \frac{(u_{n}(x)-u_{n}(u))(\psi_{R}(x)-\psi_{R}(y))}{|x-y|^{N+2s}} u_{n}(y)\, dx dy,
\end{align*}
which together with \eqref{Kbound} yields
\begin{align*}
\frac{1}{2}\int_{\R^{N}\setminus B_{R/2}} |(-\Delta)^{\frac{s}{2}}u_{n}|^{2}+V(\e x)u_{n}^{2}\, dx\leq o_{n}(1)-\iint_{\R^{2N}} \frac{(u_{n}(x)-u_{n}(u))(\psi_{R}(x)-\psi_{R}(y))}{|x-y|^{N+2s}} u_{n}(y)\, dx dy.
\end{align*}
Now, we note that the H\"older inequality and the boundedness of $(u_{n})$ imply that
\begin{align*}
&\left|\iint_{\R^{2N}} \frac{(u_{n}(x)-u_{n}(u))(\psi_{R}(x)-\psi_{R}(y))}{|x-y|^{N+2s}} u_{n}(y)\, dx dy\right| \\
&\leq \left( \iint_{\R^{2N}} \frac{|u_{n}(x)-u_{n}(y)|^{2}}{|x-y|^{N+2s}}\, dx dy \right)^{\frac{1}{2}} \left( \iint_{\R^{2N}} \frac{|\psi_{R}(x)-\psi_{R}(y)|^{2}}{|x-y|^{N+2s}} u_{n}^{2}(y)\, dx dy \right)^{\frac{1}{2}} \\
&\leq C \left( \iint_{\R^{2N}} \frac{|\psi_{R}(x)-\psi_{R}(y)|^{2}}{|x-y|^{N+2s}} u_{n}^{2}(y)\, dx dy \right)^{\frac{1}{2}}. 
\end{align*}
Therefore, it is enough to prove that 
$$
\lim_{R\rightarrow \infty} \limsup_{n\rightarrow \infty}\iint_{\R^{2N}} \frac{|\psi_{R}(x)-\psi_{R}(y)|^{2}}{|x-y|^{N+2s}} u_{n}^{2}(y)\, dx dy=0
$$
to conclude our first claim.

Let us note that $\R^{2N}$ can be written as 
$$
\R^{2N}=((\R^{N}\setminus B_{2R})\times (\R^{N}\setminus B_{2R})) \cup ((\R^{N}\setminus B_{2R})\times B_{2R})\cup (B_{2R}\times \R^{N})=: X^{1}_{R}\cup X^{2}_{R} \cup X^{3}_{R}.
$$
Then
\begin{align}\label{Pa1}
&\iint_{\R^{2N}}\frac{|\psi_{R}(x)-\psi_{R}(y)|^{2}}{|x-y|^{N+2s}} u_{n}^{2}(x) dx dy =\iint_{X^{1}_{R}}\frac{|\psi_{R}(x)-\psi_{R}(y)|^{2}}{|x-y|^{N+2s}} u_{n}^{2}(x) dx dy \nonumber \\
&+\iint_{X^{2}_{R}}\frac{|\psi_{R}(x)-\psi_{R}(y)|^{2}}{|x-y|^{N+2s}} u_{n}^{2}(x) dx dy+
\iint_{X^{3}_{R}}\frac{|\psi_{R}(x)-\psi_{R}(y)|^{2}}{|x-y|^{N+2s}} u_{n}^{2}(x) dx dy.
\end{align}
Now, we estimate each integral in (\ref{Pa1}).
Since $\psi_{R}=1$ in $\R^{N}\setminus B_{2R}$, we have
\begin{align}\label{Pa2}
\iint_{X^{1}_{R}}\frac{u_{n}^{2}(x) |\psi_{R}(x)-\psi_{R}(y)|^{2}}{|x-y|^{N+2s}} dx dy=0.
\end{align}
Let $k>4$. Clearly, we have
\begin{equation*}
X^{2}_{R}=(\R^{N} \setminus B_{2R})\times B_{2R} = ((\R^{2N}\setminus B_{kR})\times B_{2R})\cup ((B_{kR}\setminus B_{2R})\times B_{2R}) 
\end{equation*}
Let us observe that, if $(x, y) \in (\R^{2N}\setminus B_{kR})\times B_{2R}$, then
\begin{equation*}
|x-y|\geq |x|-|y|\geq |x|-2R>\frac{|x|}{2}. 
\end{equation*}
Therefore, taking into account $0\leq \psi_{R}\leq 1$, $|\nabla \psi_{R}|\leq \frac{C}{R}$ and applying H\"older inequality, we can see
\begin{align}\label{Pa3}
&\iint_{X^{2}_{R}}\frac{u_{n}^{2}(x)|\psi_{R}(x)-\psi_{R}(y)|^{2}}{|x-y|^{N+2s}} dx dy \nonumber \\
&=\int_{\R^{2N}\setminus B_{kR}} \int_{B_{2R}} \frac{u_{n}^{2}(x)|\psi_{R}(x)-\psi_{R}(y)|^{2}}{|x-y|^{N+2s}} dx dy + \int_{B_{kR}\setminus B_{2R}} \int_{B_{2R}} \frac{u_{n}^{2}(x)|\psi_{R}(x)-\psi_{R}(y)|^{2}}{|x-y|^{N+2s}} dx dy \nonumber \\
&\leq 2^{2+N+2s} \int_{\R^{N}\setminus B_{kR}} \int_{B_{2R}} \frac{u_{n}^{2}(x)}{|x|^{N+2s}}\, dxdy+ \frac{C}{R^{2}} \int_{B_{kR}\setminus B_{2R}} \int_{B_{2R}} \frac{u_{n}^{2}(x)}{|x-y|^{N+2(s-1)}}\, dxdy \nonumber \\
&\leq CR^{N} \int_{\R^{N}\setminus B_{kR}} \frac{u_{n}^{2}(x)}{|x|^{N+2s}}\, dx + \frac{C}{R^{2}} (kR)^{2(1-s)} \int_{B_{kR}\setminus B_{2R}} u_{n}^{2}(x) dx \nonumber \\
&\leq CR^{N} \left( \int_{\R^{N}\setminus B_{kR}} |u_{n}(x)|^{2^{*}_{s}} dx \right)^{\frac{2}{2^{*}_{s}}} \left(\int_{\R^{N}\setminus B_{kR}}\frac{1}{|x|^{\frac{N^{2}}{2s} +N}}\, dx \right)^{\frac{2s}{N}} + \frac{C k^{2(1-s)}}{R^{2s}} \int_{B_{kR}\setminus B_{2R}} u_{n}^{2}(x) dx \nonumber \\
&\leq \frac{C}{k^{N}} \left( \int_{\R^{N}\setminus B_{kR}} |u_{n}(x)|^{2^{*}_{s}} dx \right)^{\frac{2}{2^{*}_{s}}} + \frac{C k^{2(1-s)}}{R^{2s}} \int_{B_{kR}\setminus B_{2R}} u_{n}^{2}(x) dx \nonumber \\
&\leq \frac{C}{k^{N}}+ \frac{C k^{2(1-s)}}{R^{2s}} \int_{B_{kR}\setminus B_{2R}} u_{n}^{2}(x) dx.
\end{align}

\noindent
Now, fix $\e\in (0,1)$, and we note that
\begin{align}\label{Ter1}
&\iint_{X^{3}_{R}} \frac{u_{n}^{2}(x) |\psi_{R}(x)- \psi_{R}(y)|^{2}}{|x-y|^{N+2s}}\, dxdy \nonumber\\
&\leq \int_{B_{2R}\setminus B_{\varepsilon R}} \int_{\R^{N}} \frac{u_{n}^{2}(x) |\psi_{R}(x)- \psi_{R}(y)|^{2}}{|x-y|^{N+2s}}\, dxdy + \int_{B_{\varepsilon R}} \int_{\R^{N}} \frac{u_{n}^{2}(x) |\psi_{R}(x)- \psi_{R}(y)|^{2}}{|x-y|^{N+2s}}\, dxdy. 
\end{align}
Let us estimate the first integral in \eqref{Ter1}. Then, 
\begin{align*}
\int_{B_{2R}\setminus B_{\varepsilon R}} \int_{\R^{N} \cap \{y: |x-y|<R\}} \frac{u_{n}^{2}(x) |\psi_{R}(x)- \psi_{R}(y)|^{2}}{|x-y|^{N+2s}}\, dxdy \leq \frac{C}{R^{2s}} \int_{B_{2R}\setminus B_{\varepsilon R}} u_{n}^{2}(x) dx
\end{align*}
and 
\begin{align*}
\int_{B_{2R}\setminus B_{\varepsilon R}} \int_{\R^{N} \cap \{y: |x-y|\geq R\}} \frac{u_{n}^{2}(x) |\psi_{R}(x)- \psi_{R}(y)|^{2}}{|x-y|^{N+2s}}\, dxdy \leq \frac{C}{R^{2s}} \int_{B_{2R}\setminus B_{\varepsilon R}} u_{n}^{2}(x) dx
\end{align*}
from which we have
\begin{align}\label{Ter2}
\int_{B_{2R}\setminus B_{\varepsilon R}} \int_{\R^{N}} \frac{u_{n}^{2}(x) |\psi_{R}(x)- \psi_{R}(y)|^{2}}{|x-y|^{N+2s}}\, dxdy \leq \frac{C}{R^{2s}} \int_{B_{2R}\setminus B_{\varepsilon R}} u_{n}^{2}(x) dx. 
\end{align}
By using the definition of $\psi_{R}$, $\e\in (0,1)$, and $\psi_{R}\leq 1$, we have 
\begin{align}\label{Ter3}
\int_{B_{\varepsilon R}} \int_{\R^{N}} \frac{u_{n}^{2}(x) |\psi_{R}(x)- \psi_{R}(y)|^{2}}{|x-y|^{N+2s}}\, dxdy &= \int_{B_{\varepsilon R}} \int_{\R^{N}\setminus B_{R}} \frac{|u_{n}(x)|^{2} |\psi_{R}(x)- \psi_{R}(y)|^{2}}{|x-y|^{N+2s}}\, dxdy\nonumber \\
&\leq 4 \int_{B_{\varepsilon R}} \int_{\R^{N}\setminus B_{R}} \frac{u_{n}^{2}(x)}{|x-y|^{N+2s}}\, dxdy\nonumber \\
&\leq C \int_{B_{\varepsilon R}} u_{n}^{2}(x) dx \int_{(1-\e)R}^{\infty} \frac{1}{r^{1+2s}} dr\nonumber \\
&=\frac{C}{[(1-\e)R]^{2s}} \int_{B_{\varepsilon R}} u_{n}^{2}(x) dx
\end{align}
where we use the fact that if $(x, y) \in B_{\varepsilon R}\times (\R^{N} \setminus B_{R})$, then $|x-y|>(1-\e)R$. \\
Taking into account \eqref{Ter1}, \eqref{Ter2} and \eqref{Ter3} we deduce 
\begin{align}\label{Pa4}
\iint_{X^{3}_{R}} &\frac{u_{n}^{2}(x) |\psi_{R}(x)- \psi_{R}(y)|^{2}}{|x-y|^{N+2s}}\, dxdy \nonumber \\
&\leq \frac{C}{R^{2s}} \int_{B_{2R}\setminus B_{\varepsilon R}} |u_{n}(x)|^{2} dx + \frac{C}{[(1-\e)R]^{2s}} \int_{B_{\varepsilon R}} u_{n}^{2}(x) dx. 
\end{align}
Putting together \eqref{Pa1}, \eqref{Pa2}, \eqref{Pa3} and \eqref{Pa4}, we can infer 
\begin{align}\label{Pa5}
\iint_{\R^{2N}} &\frac{u_{n}^{2}(x) |\psi_{R}(x)-\psi_{R}(y)|^{2}}{|x-y|^{N+2s}}\, dxdy \nonumber \\
&\leq \frac{C}{k^{N}} + \frac{Ck^{2(1-s)}}{R^{2s}} \int_{B_{kR}\setminus B_{2R}} u_{n}^{2}(x) dx + \frac{C}{R^{2s}} \int_{B_{2R}\setminus B_{\varepsilon R}} |u_{n}(x)|^{2} dx + \frac{C}{[(1-\e)R]^{2s}}\int_{B_{\varepsilon R}} u_{n}^{2}(x) dx. 
\end{align}
Since $(u_{n})$ is bounded in $H^{s}(\R^{N})$, we may assume that $u_{n}\rightarrow u$ in $L^{2}_{loc}(\R^{N})$ for some $u\in H^{s}(\R^{N})$. Then, taking the limit as $n\rightarrow \infty$ in \eqref{Pa5}, we have
\begin{align*}
&\limsup_{n\rightarrow \infty} \iint_{\R^{2N}} \frac{|u_{n}(x)|^{2} |\psi_{R}(x)- \psi_{R}(y)|^{2}}{|x-y|^{N+2s}}\, dxdy\\
&\leq \frac{C}{k^{N}} + \frac{Ck^{2(1-s)}}{R^{2s}} \int_{B_{kR}\setminus B_{2R}} |u(x)|^{2} dx + \frac{C}{R^{2s}} \int_{B_{2R}\setminus B_{\varepsilon R}} |u(x)|^{2} dx + \frac{C}{[(1-\e)R]^{2s}}\int_{B_{\varepsilon R}} |u(x)|^{2} dx \\
&\leq \frac{C}{k^{N}} + Ck^{2} \left( \int_{B_{kR}\setminus B_{2R}} |u(x)|^{2^{*}_{s}} dx\right)^{\frac{2}{2^{*}_{s}}} + C\left(\int_{B_{2R}\setminus B_{\varepsilon R}} |u(x)|^{2^{*}_{s}} dx\right)^{\frac{2}{2^{*}_{s}}} + C\left( \frac{\e}{1-\e}\right)^{2s} \left(\int_{B_{\varepsilon R}} |u(x)|^{2^{*}_{s}} dx\right)^{\frac{2}{2^{*}_{s}}}, 
\end{align*}
where in the last passage we use H\"older inequality. \\
Since $u\in L^{2^{*}_{s}}(\R^{N})$, $k>4$ and $\e \in (0,1)$, we obtain
\begin{align*}
\limsup_{R\rightarrow \infty} \int_{B_{kR}\setminus B_{2R}} |u(x)|^{2^{*}_{s}} dx = \limsup_{R\rightarrow \infty} \int_{B_{2R}\setminus B_{\varepsilon R}} |u(x)|^{2^{*}_{s}} dx = 0. 
\end{align*}
Choosing $\e= \frac{1}{k}$, we get
\begin{align*}
&\limsup_{R\rightarrow \infty} \limsup_{n\rightarrow \infty} \iint_{\R^{2N}} \frac{u_{n}^{2}(x) |\psi_{R}(x)- \psi_{R}(y)|^{2}}{|x-y|^{N+2s}}\, dxdy\\
&\leq \lim_{k\rightarrow \infty} \limsup_{R\rightarrow \infty} \Bigl[\, \frac{C}{k^{N}} + Ck^{2} \left( \int_{B_{kR}\setminus B_{2R}} |u(x)|^{2^{*}_{s}} dx\right)^{\frac{2}{2^{*}_{s}}} + C\left(\int_{B_{2R}\setminus B_{\frac{1}{k} R}} |u(x)|^{2^{*}_{s}} dx\right)^{\frac{2}{2^{*}_{s}}} \\
&+ C\left(\frac{1}{k-1}\right)^{2s} \left(\int_{B_{\frac{1}{k} R}} |u(x)|^{2^{*}_{s}} dx\right)^{\frac{2}{2^{*}_{s}}}\, \Bigr]\\
&\leq \lim_{k\rightarrow \infty} \frac{C}{k^{N}} + C\left(\frac{1}{k-1}\right)^{2s} \left(\int_{\R^{N}} |u(x)|^{2^{*}_{s}} dx \right)^{\frac{2}{2^{*}_{s}}}= 0.
\end{align*}

\noindent
{\bf Step $2$}: Let us prove that $u_{n}\rightarrow u$ in $H^{s}_{\e}$ as $n\rightarrow \infty$.\\
Set $\Psi_{n}=\|u_{n}-u\|^{2}_{\e}$ and we observe that 
\begin{equation}\label{psin}
\Psi_{n}=\langle J'_{\e}(u_{n}),u_{n}\rangle-\langle J'_{\e}(u_{n}), u\rangle+\int_{\R^{N}} \left(\frac{1}{|x|^{\mu}}*G(\e x, u_{n}) \right) g(\e x, u_{n})(u_{n}-u)\, dx+o_{n}(1).
\end{equation}
Let us note that $\langle J'_{\e}(u_{n}),u_{n}\rangle=\langle J'_{\e}(u_{n}),u\rangle=o_{n}(1)$, so in view of \eqref{psin}, we need to show that 
$$
\int_{\R^{N}} \left( \frac{1}{|x|^{\mu}}*G(\e x, u_{n}) \right) g(\e x, u_{n})(u_{n}-u)\, dx=o_{n}(1),
$$
to infer that $\Psi_{n}\rightarrow 0$ as $n\rightarrow \infty$.\\
We observe that $G(\e x, u_{n})$ is bounded in $L^{\frac{2N}{2N-\mu}}(\R^{N})$ (since $q<\frac{2^{*}_{s}}{2}\left(2-\frac{\mu}{N}\right)$), $u_{n}\rightarrow u$ a.e. in $\R^{N}$, and $G$ is continuous, so we deduce that 
\begin{equation}\label{Gconv}
G(\e x, u_{n})\rightharpoonup G(\e x, u) \mbox{ in } L^{\frac{2N}{2N-\mu}}(\R^{N}).
\end{equation}
In virtue of Theorem \ref{HLS}, we know that the convolution term
$$
\frac{1}{|x|^{\mu}}*h(x)\in L^{\frac{2N}{\mu}}(\R^{N}) \mbox{ for all } h\in L^{\frac{2N}{2N-\mu}}(\R^{N})
$$
is a linear bounded operator from $L^{\frac{2N}{2N-\mu}}(\R^{N})$ to $L^{\frac{2N}{\mu}}(\R^{N})$,
so we can see that
\begin{equation}\label{wconv}
\tilde{K}_{\e}(u_{n}) =\frac{1}{|x|^{\mu}}*G(\e x, u_{n})\rightharpoonup \frac{1}{|x|^{\mu}}*G(\e x, u)=\tilde{K}_{\e}(u)  \mbox{ in } L^{\frac{2N}{\mu}}(\R^{N}).
\end{equation}
Since $g$ has a subcritical growth, by using Theorem \ref{Sembedding} and \eqref{wconv}, we obtain
\begin{equation}\label{Me1limit}
\lim_{n\rightarrow \infty} \int_{B_{R}} \tilde{K}_{\e}(u_{n}) g(\e x, u_{n})(u_{n}-u)  \, dx=0.
\end{equation}
From the growth assumption and the boundedness of $\tilde{K}_{\e}(u_{n})$ we obtain
\begin{equation*}
\int_{\R^{N}\setminus B_{R}} \tilde{K}_{\e}(u_{n}) |g(\e x, u_{n})u_{n}| \, dx\leq C_{1}\int_{\R^{N}\setminus B_{R}} u_{n}^{2}\, dx.
\end{equation*}
By the Step $1$ and Theorem \ref{Sembedding}, for any $\eta>0$ there exists $R_{\eta}>0$ such that
$$
\limsup_{n\rightarrow \infty} \int_{\R^{N}\setminus B_{R}} \tilde{K}_{\e}(u_{n}) |g(\e x, u_{n})u_{n}| \, dx\leq C_{2}\eta. 
$$
In similar way, from H\"older inequality, we can see that
$$
\limsup_{n\rightarrow \infty} \int_{\R^{N}\setminus B_{R}} \tilde{K}_{\e}(u_{n}) |g(\e x, u_{n})u| \, dx\leq C_{3}\eta. 
$$
Taking into account the above limits we can infer that  
$$
\lim_{n\rightarrow \infty} \int_{\R^{N}} \tilde{K}_{\e}(u_{n}) g(\e x, u_{n})(u_{n}-u) \, dx=0. 
$$
\end{proof}

\noindent
Finally, we prove the following result:
\begin{lem}\label{lemma2.10}
The functional $\psi_{\e}$ satisfies the $(PS)_{c}$ on $\mathbb{S}_{\e}$ for any $c\in [c_{\e}, \kappa]$.
\end{lem}
\begin{proof}
Let $(u_{n})\subset \mathbb{S}_{\e}$ be a $(PS)_{c}$ sequence for $\psi_{\e}$. Then $\psi_{\e}(u_{n})\rightarrow c$ and $\|\psi'_{\e}(u_{n})\|_{*}\rightarrow 0$, where $\|\cdot\|_{*}$  denotes the norm in the dual space of $(T_{u_{n}}\mathbb{S}_{\e})^{*}$. By using Proposition \ref{propz2}-$(c)$, we can infer that $(m_{\e}(u_{n}))$ is a $(PS)_{c}$ sequence for $J_{\e}$. In view of Lemma \ref{PSc}, we can see that, up to a subsequence, there exists $u\in \mathbb{S}_{\e}$ such that $m_{\e}(u_{n})\rightarrow m_{\e}(u)$ in $H^{s}_{\e}$.
From Lemma \ref{lemz1}-$(c)$, we conclude that $u_{n}\rightarrow u$ in $\mathbb{S}_{\e}$.

\end{proof}

\section{The autonomous problem}
In this section we deal with the limit problem associated to \eqref{R}, namely
\begin{equation}\label{AP}
(-\Delta)^{s} u + V_{0}u = \left(\frac{1}{|x|^{\mu}}*F(u)\right)f(u) \mbox{ in } \R^{N}. 
\end{equation}
In what follows, we denote the above problem with $(P_{V_{0}})$.\\
The functional $J_{V_{0}}: H^{s}_{0}\rightarrow \R$ associated to the above problem is given by
$$
J_{V_{0}}(u)=\frac{1}{2}\|u\|^{2}_{V_{0}}-\Sigma_{0}(u), 
$$
where $H_{0}^{s}$ is the space $H^{s}(\R^{N})$ endowed with the norm 
$$
\|u\|^{2}_{V_{0}}=[u]^{2}+\int_{\R^{N}} V_{0} u^{2}\,dx,
$$
and
$$
\Sigma_{0}(u)=\frac{1}{2}\int_{\R^{N}} \left(\frac{1}{|x|^{\mu}}*F(u)\right)F(u)\, dx.
$$
Let us consider the following Nehari manifold
$$
\mathcal{N}_{V_{0}}=\{u\in H^{s}_{0}\setminus\{0\}: \langle J'_{V_{0}}(u), u\rangle=0 \}
$$ 
and let us denote by $\mathbb{S}_{0}$ the unit sphere in $H^{s}_{0}$.
Arguing as in the proofs of Lemma $\ref{lemz1}$ and Proposition $\ref{propz2}$, we can see that the following results hold.
\begin{lem}\label{lemz1e0}
Suppose that $f$ verifies $(f_1)$-$(f_4)$. \\
Then, the following facts hold true:
\begin{compactenum}
\item[$(a)$] 
For any $u\in H^{s}_{0}\setminus\{0\}$, let $h_{u}: \R_{+} \rightarrow \R$ be defined by $h_{u}(t):= J_{V_{0}}(t u)$. Then, there is a unique $t_{u}>0$ such that $h_{u}'(t)>0$ in $(0, t_{u})$ and $h_{u}'(t)<0$ in $(t_{u}, +\infty)$.
\item[$(b)$] There is $\tau>0$, independent on $u$, such that $t_{u}\geq \tau$ for every $u\in \mathbb{S}_{0}$. Moreover, for each compact set $\mathcal{W}\subset \mathbb{S}_{0}$, there is $C_{\mathcal{W}}>0$ such that $t_{u}\leq C_{\mathcal{W}}$ for every $u\in \mathcal{W}$.
\item[$(c)$] The map $\hat{m}_{0}: H^{s}_{0}\setminus\{0\}\rightarrow \mathcal{N}_{V_{0}}$ given by $\hat{m}_{0}(u):=t_{u}u$ is continuous and $m_{0}:= \hat{m}|_{\mathbb{S}_{0}}$ is a homeomorphism between $\mathbb{S}_{0}$ and $\mathcal{N}_{V_{0}}$. Moreover, $m^{-1}_{0}(u)= \frac{u}{\|u\|_{V_{0}}}$. 
\end{compactenum}
\end{lem}

\noindent
Let us define the maps $\hat{\psi}_{0}: H^{s}_{0}\setminus\{0\} \rightarrow \R$ by $\hat{\psi}_{0}(u):= J_{V_{0}}(\hat{m}_{0}(u))$, and $\psi:=\hat{\psi}_{0}|_{\mathbb{S}_{0}}$. 
Then we have
\begin{prop}\label{propz2e0}
Suppose that $f$ verifies $(f_1)$-$(f_4)$. Then, one has:
\begin{compactenum}
\item[$(a)$] $\hat{\psi}_{0}\in C^{1}(H^{s}_{0}\setminus\{0\}, \R)$ and
\begin{equation*}
\langle \hat{\psi}_{0}'(u), v \rangle=\frac{\|\hat{m}_{0}(u)\|_{V_{0}}}{\|u\|_{V_{0}}} \langle J_{\e}'(\hat{m}_{0}(u)), v \rangle \,,
\end{equation*}
for every $u\in H^{s}_{0}\setminus\{0\}$ and $v\in H^{s}_{0}$;
\item[$(b)$] $\psi_{0} \in C^{1}(\mathbb{S}_{0}, \R)$ and $\langle \psi'_{0}(u), v \rangle = \|m_{0}(u)\|_{V_{0}} \langle J_{V_{0}}'(m_{0}(u)), v \rangle$, for every $v\in T_{u}\mathbb{S}_{0}$.
\item[$(c)$] If $(u_{n})$ is a $(PS)_{d}$ sequence for $\psi_{0}$, then $(m_{0}(u_{n}))$ is a $(PS)_{d}$ sequence for $J_{V_{0}}$. Moreover, if $(u_{n})\subset \mathcal{N}_{V_{0}}$ is a bounded $(PS)_{d}$ sequence for $J_{V_{0}}$, then $(m_{0}^{-1} (u_{n}))$ is a $(PS)_{d}$ sequence for the functional $\psi_{0}$;
\item[$(d)$] $u$ is a critical point of $\psi_{0}$ if and only if $m_{0}(u)$ is a nontrivial critical point for $J_{V_{0}}$. Moreover, the corresponding critical values coincide and
\begin{equation*}
\inf_{u\in\mathbb{S}_{0}} \psi_{0}(u) = \inf_{u\in\mathcal{N}_{V_{0}}} J_{V_{0}}(u).
\end{equation*}
\end{compactenum}
\end{prop}

\noindent
Moreover, we have the following characterization of the infimum of $J_{0}$ on $\mathcal{N}_{V_{0}}$
\begin{equation}\label{cv0c}
c_{V_{0}}= \inf_{u\in\mathcal{N}_{V_{0}}} J_{V_{0}}(u)=\inf_{u\in H^{s}_{0}\setminus\{0\}} \max_{t>0} J_{V_{0}}(t u)=\inf_{u\in \mathbb{S}_{0}} \max_{t>0} J_{V_{0}}(t u).
\end{equation}

\noindent
The next Lemma allows us to assume that the weak limit of a $(PS)_{c}$ sequence is nontrivial.
\begin{lem}\label{HZL}
Let $(u_{n})\subset H^{s}_{0}$ be a $(PS)_{c}$ sequence for $J_{V_{0}}$ and such that $u_{n}\rightharpoonup 0$. Then, only one of the following alternatives holds.
\begin{compactenum}[$(a)$]
\item $u_{n}\rightarrow 0$ in $H^{s}_{0}$, or
\item there exists a sequence $(\tilde{y}_{n})\subset \R^{N}$, and constants $R>0$ and $\gamma>0$ such that
\begin{equation*}
\liminf_{n\rightarrow \infty}\int_{B_{R}(\tilde{y}_{n})} |u_{n}|^{2} \, dx\geq \gamma>0.
\end{equation*}
\end{compactenum}
\end{lem}
\begin{proof}
Suppose that $(b)$ does not hold. Then, for all $R>0$, we have
$$
\lim_{n\rightarrow \infty}\sup_{y\in \R^{N}}\int_{B_{R}(y)} |u_{n}|^{2} \, dx=0.
$$
Since we know that $(u_{n})$ is bounded in $H^{s}_{0}$, we can use Lemma \ref{lions lemma} to deduce that $u_{n}\rightarrow 0$ in $L^{q}(\R^{N})$ for any $q\in (2, 2^{*}_{s})$. By using $(f_1)$-$(f_2)$, we know that for all $\eta>0$ there exists $C_{\eta}>0$ such that
$$
|F(t)|\leq \eta|t|^{2}+C_{\eta}|t|^{q},
$$
so, applying Hardy-Littlewood-Sobolev inequality, we get
$$
\int_{\R^{N}} \left(\frac{1}{|x|^{\mu}}*F(u_{n}) \right) f(u_{n})u_{n}\, dx=o_{n}(1).
$$
Taking into account $\langle J'_{0}(u_{n}), u_{n}\rangle=o_{n}(1)$, we can  infer that $\|u_{n}\|_{V_{0}}\rightarrow 0$ as $n\rightarrow \infty$. 
\end{proof}

\noindent
Now, we prove the following result for the autonomous problem.
\begin{lem}\label{BBMP}
Let $(w_{n})\subset H^{s}_{0}$ be a $(PS)_{c_{V_{0}}}$ sequence for $J_{V_{0}}$. Then the problem $(P_{V_{0}})$ has a positive ground state.
\end{lem}
\begin{proof}
Arguing as in Lemma \ref{MPG}, we can see that $J_{0}$ has a mountain pass geometry. As a consequence of the mountain pass theorem without the (PS) condition (see \cite{W}), there exists a Palais-Smale sequence $(u_{n})\subset H^{s}_{0}$ such that
$$
 J_{V_{0}}(u_{n})\rightarrow c_{V_{0}} \mbox{ and } J'_{V_{0}}(u_{n})\rightarrow 0.
$$
Since
$$
J_{V_{0}}(u_{n})-\frac{1}{4}\langle J'_{V_{0}}(u_{n}), u_{n}\rangle\geq \frac{1}{4}\|u_{n}\|^{2}_{V_{0}},
$$
it is easy to deduce that $(u_{n})$ is bounded in $H^{s}_{0}$. By using $(f_{1})$-$(f_{2})$, we know that $\|u\|_{V_{0}}\geq r$ for all $u\in \mathcal{N}_{V_{0}}$. Then, arguing as in the proof of Lemma \ref{HZL}, we can see that there exists a sequence $(y_{n})\subset \R^{N}$, and constants $R>0$ and $\gamma>0$ such that
\begin{equation*}
\liminf_{n\rightarrow \infty}\int_{B_{R}(y_{n})} |u_{n}|^{2} \, dx\geq \gamma>0.
\end{equation*}
Set $v_{n}=u_{n}(\cdot-y_{n})$. Since $J_{V_{0}}$ and $J'_{V_{0}}$ are both invariant by translation, it holds that
$$
J'_{V_{0}}(v_{n})\rightarrow 0 \mbox{ and } J_{V_{0}}(v_{n})\rightarrow c_{V_{0}}.
$$
We observe that $(v_{n})$ is also bounded in $H^{s}_{0}$, so we may assume that $v_{n}\rightharpoonup v$ in  $H^{s}_{0}$, for some $v\neq 0$. Now, we show that $v$ is a weak solution to $(P_{V_{0}})$.
Fix $\varphi\in C^{\infty}_{0}(\R^{N})$. Recalling that $(v_{n})$ is bounded, we can argue as in the proof of Lemma \ref{lemK} to deduce that $\|\frac{1}{|x|^{\mu}}*F(v_{n})\|_{L^{\infty}(\R^{N})}\leq C$ for any $n\in \N$. Then, using the fact that $f$ has subcritical growth and $v_{n}\rightarrow v$ in $L^{r}_{loc}(\R^{N})$ for any $r\in [1, 2^{*}_{s})$, we can see that the Dominated Convergence Theorem gives 
$$
\int_{\R^{N}} \left(\frac{1}{|x|^{\mu}}*F(v_{n})\right)f(v_{n})\varphi\, dx\rightarrow \int_{\R^{N}} \left(\frac{1}{|x|^{\mu}}*F(v)\right)f(v)\varphi\, dx.
$$
This combined with the weak convergence of $(v_{n})$ yields
$$
o_{n}(1)=\langle J'_{V_{0}}(v_{n}), \varphi\rangle\rightarrow \langle J'_{V_{0}}(v), \varphi\rangle.
$$
From the density of $C^{\infty}_{0}(\R^{N})$ in $H^{s}_{0}$, we get $ \langle J'_{V_{0}}(v), \varphi\rangle=0$ for all $\varphi\in H^{s}_{0}$.
In particular, $v\in \mathcal{N}_{V_{0}}$. Using the definition of $c_{V_{0}}$ together with Fatou's Lemma, we also deduce that $J_{V_{0}}(v)=c_{V_{0}}$. 

Now, recalling that $f(t)=0$ for $t\leq 0$ and $(x-y)(x^{-}-y^{-})\geq |x^{-}-y^{-}|^{2}$ for all $x, y\in \R$, it is easy to deduce that $\langle J'_{V_{0}}(v), v^{-}\rangle=0$ implies that $v\geq 0$ in $\R^{N}$. \\
Proceeding as in the proof of Lemma \ref{lemK}, we can see that $K(x):=\frac{1}{|x|^{\mu}}*F(v)$ is bounded in $\R^{N}$, so similar arguments developed in Lemma \ref{lemMoser} below, allow us to deduce that $v\in L^{\infty}(\R^{N})$. Since $f$ has subcritical growth and $K(x)$ is bounded, we can see that $K(x)f(v)\in L^{\infty}(\R^{N})$, so we can apply Proposition $2.9$ in \cite{S} to infer that $v\in C^{0, \alpha}(\R^{N})$ for some $\alpha\in (0, 1)$. Using the Harnack inequality \cite{S}, we can conclude that $v>0$ in $\R^{N}$.
\end{proof}

\noindent
The next result is a compactness result on autonomous problem which we will use later.
\begin{lem}\label{Ekeland}
Let $(\tilde{v}_{n})\subset \mathcal{N}_{V_{0}}$ be such that $J_{0}(\tilde{v}_{n})\rightarrow c_{V_{0}}$. Then  $(\tilde{v}_{n})$ has a convergent subsequence in $H^{s}_{0}$.
\end{lem}
\begin{proof}
Since $(\tilde{v}_{n})\subset \mathcal{N}_{V_{0}}$ and $J_{V_{0}}(\tilde{v}_{n})\rightarrow c_{V_{0}}$, we can apply Lemma \ref{lemz1e0}-(c) and Proposition \ref{propz2e0}-(d) to infer that
$$
w_{n}=m_{0}^{-1}(\tilde{v}_{n})=\frac{\tilde{v}_{n}}{\|\tilde{v}_{n}\|_{V_{0}}}\in \mathbb{S}_{0}
$$
and
$$
\psi_{0}(w_{n})=J_{V_{0}}(\tilde{v}_{n})\rightarrow c_{V_{0}}=\inf_{v\in \mathbb{S}_{0}}\psi_{0}(v).
$$
Hence, by using the Ekeland's variational principle \cite{Ekeland}, we can find $(\tilde{w}_{n})\subset \mathbb{S}_{0}$ such that $(\tilde{w}_{n})$ is a $(PS)_{c_{V_{0}}}$ sequence for $\psi_{0}$ on $\mathbb{S}_{0}$ and $\|\tilde{w}_{n}-w_{n}\|_{V_{0}}=o_{n}(1)$.
From Proposition \ref{propz2e0}-(c), we can deduce that $m_{0}(\tilde{w}_{n})$ is a $(PS)_{c_{V_{0}}}$ sequence of $J_{0}$. 
By applying Lemma \ref{BBMP}, it follows that there exists $\tilde{w}\in \mathbb{S}_{0}$ such that $m_{0}(\tilde{w}_{n})\rightarrow m_{0}(\tilde{w})$ in $H^{s}_{0}$. This fact, together with Lemma \ref{lemz1e0}-(c), and $\|\tilde{w}_{n}-w_{n}\|_{V_{0}}=o_{n}(1)$, allow us to conclude that  $\tilde{v}_{n}\rightarrow \tilde{v}$ in $H^{s}_{0}$.
\end{proof}

\section{Multiplicity results}
In order to study the multiplicity of solutions to (\ref{P}), we need introduce some useful tools.\\
Let us consider $\delta>0$ such that $M_{\delta}\subset \Lambda$, where
$$
M_{\delta}=\{x\in \R^{N}: dist(x, M)\leq \delta\}.
$$
and  $\eta\in C^{\infty}_{0}(\R_{+}, [0, 1])$ satisfying $\eta(t)=1$ if $0\leq t\leq \frac{\delta}{2}$ and $\eta(t)=0$ if $t\geq \delta$.
 
For any $y\in M$, we define
$$
\Psi_{\e, y}(x)=\eta(|\e x-y|) w\left(\frac{\e x-y}{\e}\right)
$$
where $w$ is a positive ground state solution for $J_{V_{0}}$ (by Lemma \ref{BBMP}).\\
Let us denote by $t_{\e}>0$ the unique positive number verifying 
$$
\max_{t\geq 0} J_{\e}(t \Psi_{\e, y})=J_{\e}(t_{\e} \Psi_{\e, y}).
$$
Finally, we consider $\Phi_{\e}(y)=t_{\e} \Psi_{\e, y}$.

In next lemma we prove an important relationship between $\Phi_{\e}$ and the set $M$.
\begin{lem}\label{lemma3.4}
The functional $\Phi_{\e}$ satisfies the following limit
$$
\lim_{\e\rightarrow 0} J_{\e}(\Phi_{\e}(y))=c_{V_{0}} \mbox{ uniformly in } y\in M.
$$
\end{lem}
\begin{proof}
Assume by contradiction that there exist $\delta_{0}>0$, $(y_{n})\subset M$ and $\e_{n}\rightarrow 0$ such that 
\begin{equation}\label{4.41}
|J_{\e_{n}}(\Phi_{\e_{n}}(y_{n}))-c_{V_{0}}|\geq \delta_{0}.
\end{equation}
We first show that $\lim_{n\rightarrow \infty}t_{\e_{n}}<\infty$.
Let us observe that by using the change of variable $z=\frac{\e_{n}x-y_{n}}{\e_{n}}$, if $z\in B_{\frac{\delta}{\e_{n}}}(0)$, it follows that $\e_{n} z\in B_{\delta}(0)$ and $\e_{n} z+y_{n}\in B_{\delta}(y_{n})\subset M_{\delta}\subset \Lambda$. 

Since $G=F$ on $\Lambda$, we can see that 
\begin{align}
J_{\e_{n}}(\Phi_{\e_{n}}(z_{n}))&=\frac{t_{\e_{n}}^{2}}{2}\int_{\R^{N}} |(-\Delta)^{\frac{s}{2}}(\eta(|\e_{n} z|)w(z))|^{2}\, dz+\frac{t_{\e_{n}}^{2}}{2}\int_{\R^{N}} V(\e_{n} z+y_{n}) (\eta(|\e_{n} z|) w(z))^{2}\, dz \nonumber\\
&-\Sigma_{0}(t_{\e_{n}}\eta(|\e_{n} z|)w(z)).
\end{align}
In view of the Dominated Convergence Theorem and Lemma $5$ in \cite{PP}, we can see that
$$
\lim_{n\rightarrow \infty} \|\Psi_{\e_{n}, y_{n}}\|_{\e_{n}}=\|w\|_{V_{0}}
$$
and
$$
\lim_{n\rightarrow \infty} \Sigma_{0}(\Psi_{\e_{n}, y_{n}})=\Sigma_{0}(w).
$$

By using $t_{\e_{n}}\Psi_{\e_{n}, y_{n}}\in \mathcal{N}_{\e_{n}}$ and the assumptions on $f$, it is easy to prove that $t_{\e_{n}}\rightarrow t_{0}>0$. Moreover, being
\begin{equation}\label{3.9}
t_{\e_{n}}^{2}\|\Psi_{\e_{n}, y_{n}}\|_{\e_{n}}^{2}=\int_{\R^{N}}\int_{\R^{N}} \frac{F(t_{\e_{n}}\Psi_{\e_{n}, y_{n}})f(t_{\e_{n}}\Psi_{\e_{n}, y_{n}})t_{\e_{n}}\Psi_{\e_{n}, y_{n}}}{|x-y|^{\mu}}
\end{equation}
we can deduce that
$$
\|w\|^{2}_{V_{0}}=\lim_{n\rightarrow \infty} \int_{\R^{N}}\int_{\R^{N}} \frac{F(t_{\e_{n}}\Psi_{\e_{n}, y_{n}})f(t_{\e_{n}}\Psi_{\e_{n}, y_{n}})t_{\e_{n}}\Psi_{\e_{n}, y_{n}}}{t_{\e_{n}}^{2}|x-y|^{\mu}}.
$$
Taking into account that $w$ is a ground state to $(P_{V_{0}})$ and using $(f_4)$, we can conclude that  $t_{\e_{n}}\rightarrow 1$.
As a consequence
$$
\lim_{n\rightarrow \infty} \Sigma_{0}(t_{\e_{n}}\eta(|\e_{n} z|)w(z))=\Sigma_{0}(w)
$$
and this yields
$$
\lim_{n\rightarrow \infty} J_{\e_{n}}(\Phi_{\e_{n}}(y_{n}))=J_{0}(w)=c_{V_{0}},
$$
which contradicts (\ref{4.41}).

\end{proof}

\noindent
Now, we consider $\delta>0$ such that $M_{\delta}\subset \Lambda$, and choose $\rho=\rho(\delta)>0$ such that $M_{\delta}\subset B_{\delta}(0)$.
We define $\Upsilon: \R^{N}\rightarrow \R^{N}$ by setting $\Upsilon(x)=x$ for $|x|\leq \rho$ and $\Upsilon(x)=\frac{\rho x}{|x|}$ for $|x|\geq \rho$.
Then we define the barycenter map $\beta_{\e}: \mathcal{N}_{\e}\rightarrow \R^{N}$ given by
$$
\beta_{\e}(u)=\frac{\int_{\R^{N}} \Upsilon(\e x) u^{2}(x)\, dx}{\int_{\R^{N}} u^{2}(x) \,dx}.
$$

\begin{lem}\label{lemma3.5}
The function $\beta_{\e}$ verifies the following limit
$$
\lim_{\e \rightarrow 0} \beta_{\e}(\Phi_{\e}(y))=y \mbox{ uniformly in } y\in M.
$$
\end{lem}
\begin{proof}
Suppose by contradiction that there exists $\delta_{0}>0$, $(y_{n})\subset M$ and $\e_{n}\rightarrow 0$ such that 
\begin{equation}\label{4.4}
|\beta_{\e_{n}}(\Phi_{\e_{n}}(y_{n}))-y_{n}|\geq \delta_{0}.
\end{equation}
By using the definitions of $\Phi_{\e_{n}}(y_{n})$, $\beta_{\e_{n}}$ and $\eta$, and using a change of variable, we can see that 
$$
\beta_{\e_{n}}(\Psi_{\e_{n}}(y_{n}))=y_{n}+\frac{\int_{\R^{N}}[\Upsilon(\e_{n}z+y_{n})-y_{n}] |\eta(|\e_{n}z|) w(z)|^{2} \, dz}{\int_{\R^{N}} |\eta(|\e_{n}z|) w(z)|^{2}\, dz}.
$$
Taking into account $(y_{n})\subset M\subset B_{\rho}$ and the Dominated Convergence Theorem, we can infer that 
$$
|\beta_{\e_{n}}(\Phi_{\e_{n}}(y_{n}))-y_{n}|=o_{n}(1)
$$
which contradicts (\ref{4.4}).

\end{proof}

\noindent
The next compactness result will be fundamental to show that the solutions of the modified problem are solutions of the original problem.
\begin{lem}\label{prop3.3}
Let $\e_{n}\rightarrow 0$ and $(u_{n})\subset \mathcal{N}_{\e_{n}}$ be such that $J_{\e_{n}}(u_{n})\rightarrow c_{V_{0}}$. Then there exists $(\tilde{y}_{n})\subset \R^{N}$ such that $v_{n}(x)=u_{n}(x+\tilde{y}_{n})$ has a convergent subsequence in $H^{s}(\R^{N})$. Moreover, up to a subsequence, $y_{n}=\e_{n} \tilde{y}_{n}\rightarrow y_{0}\in M$.
\end{lem}
\begin{proof}
Since $\langle J'_{\e_{n}}(u_{n}), u_{n}\rangle=0$ and $J_{\e_{n}}(u_{n})\rightarrow c_{V_{0}}$, we can see that $(u_{n})$ is bounded in $H^{s}_{\e_{n}}$. Note that $c_{V_{0}}>0$, and since $\|u_{n}\|_{\e_{n}}\rightarrow 0$ would imply $J_{\e_{n}}(u_{n})\rightarrow 0$, we can argue as in Lemma \ref{HZL} to obtain 
a sequence $(\tilde{y}_{n})\subset \R^{N}$, and constants $R>0$ and $\gamma>0$ such that
\begin{equation}\label{sacchi}
\liminf_{n\rightarrow \infty}\int_{B_{R}(\tilde{y}_{n})} |u_{n}|^{2} \, dx\geq \gamma>0.
\end{equation}
Now, we set $v_{n}(x)=u_{n}(x+\tilde{y}_{n})$. Then, $(v_{n})$ is bounded in $H^{s}_{0}$, and we may assume that 
$v_{n}\rightharpoonup v\not\equiv 0$ in $H^{s}_{0}$  as $n\rightarrow \infty$.
Fix $t_{n}>0$ such that $\tilde{v}_{n}=t_{n} v_{n}\in \mathcal{N}_{V_{0}}$. Since $u_{n}\in \mathcal{N}_{\e_{n}}$, we can see that 
$$
c_{V_{0}}\leq J_{V_{0}}(\tilde{v}_{n})= J_{V_{0}}(t_{n}u_{n})\leq J_{\e_{n}}(t_{n}u_{n})\leq J_{\e_{n}}(u_{n})= c_{V_{0}}+o_{n}(1)
$$
which gives $J_{V_{0}}(\tilde{v}_{n})\rightarrow c_{V_{0}}$. 
In particular, we get $\tilde{v}_{n}\rightharpoonup \tilde{v}$ in $H^{s}_{0}$ and $t_{n}\rightarrow t^{*}>0$. Then, from the uniqueness of the weak limit, we have $\tilde{v}=t^{*}v\not\equiv 0$. \\
By using Lemma \ref{Ekeland}, we can see that 
\begin{equation}\label{elena}
\tilde{v}_{n}\rightarrow \tilde{v} \mbox{ in } H^{s}_{0}.
\end{equation} 
In order to complete the proof of the lemma, we consider $y_{n}=\e_{n}\tilde{y}_{n}$. Our claim is to show that $(y_{n})$ admits a subsequence, still denoted by $y_{n}$, such that $y_{n}\rightarrow y_{0}$, for some $y_{0}\in M$. Firstly, we prove that $(y_{n})$ is bounded. We argue by contradiction, and we assume that, up to a subsequence, $|y_{n}|\rightarrow \infty$ as $n\rightarrow \infty$. 
Since
$$
\|u_{n}\|_{\e_{n}}^{2}=\int_{\R^{N}} \left(\frac{1}{|x|^{\mu}}*G(\e x, u_{n})\right)g(\e x, u_{n})u_{n},
$$
and $J_{\e_{n}}(u_{n})\rightarrow c_{V_{0}}$, we can see that $u_{n}\in \mathcal{B}$ for all $n$ big enough.
Then, in view of Lemma \ref{lemK}, there exists $C_{0}>0$ such that 
$$
\sup_{n\in \mathbb{N}} \left\| \frac{1}{|x|^{\mu}}*G(\e x, u_{n})\right\|_{L^{\infty}(\R^{N})}<C_{0}.
$$

Fixed $R>0$ such that $\Lambda \subset B_{R}(0)$, and assume that $|y_{n}|>2R$. Then, for all $z\in B_{\frac{R}{\e_{n}}}(0)$,
\begin{equation}\label{5.4AY}
|\e_{n}z+y_{n}|\geq |y_{n}|-|\e_{n}z|>R.
\end{equation} 
By using the change of variable $x\mapsto z+\tilde{y}_{n}$ and \eqref{5.4AY}, we deduce that
\begin{align}\label{pasq}
[v_{n}]^{2}+\int_{\R^{N}} V_{0} v_{n}^{2}\, dx &\leq C_{0}\int_{\R^{N}} g(\e_{n} z+y_{n}, v_{n}) v_{n} \, dx \nonumber\\
&\leq C_{0}\int_{B_{\frac{R}{\e_{n}}}(0)} \tilde{f}(v_{n}) v_{n} \, dx+C_{0}\int_{\R^{N}\setminus B_{\frac{R}{\e_{n}}}(0)} f(v_{n}) v_{n} \, dx.
\end{align}
Then, by using the fact that $v_{n}\rightarrow v$ in $H^{s}_{0}$ as $n\rightarrow \infty$ and that $\tilde{f}(t)\leq \frac{V_{0}}{\ell_{0}}t$, we can see that (\ref{pasq}) implies that
$$
[v_{n}]^{2}+\int_{\R^{N}} V_{0} v_{n}^{2}\, dx=o_{n}(1),
$$
that is $v_{n}\rightarrow 0$ in $H^{s}_{0}$, which is a contradiction. Therefore, $(y_{n})$ is bounded, and we may assume that $y_{n}\rightarrow y_{0}\in \R^{N}$. Clearly, if $y_{0}\notin \overline{\Lambda}$, then we can argue as before and we deduce that $v_{n}\rightarrow 0$ in $H^{s}_{0}$, which is impossible. Hence $y_{0}\in \overline{\Lambda}$. Now, we note that if $V(y_{0})=V_{0}$, then we can infer that $y_{0}\notin \partial \Lambda$ in view of $(V_2)$, and then $y_{0}\in M$. Therefore, in the next step, we show that $V(y_{0})=V_{0}$. Suppose by contradiction that $V(y_{0})>V_{0}$.
Then, by using $\tilde{v}_{n}\rightarrow \tilde{v}$ in $H^{s}_{0}$ and Fatou's Lemma, we get 
\begin{align*}
c_{V_{0}}=J_{V_{0}}(\tilde{v})&<\frac{1}{2} \left([\tilde{v}]^{2}+\int_{\R^{N}} V(y_{0})\tilde{v}^{2}\right)-\Sigma_{0}(\tilde{v}) \\
&\leq \liminf_{n\rightarrow \infty}\left[\frac{1}{2}[\tilde{v}_{n}]^{2}+\frac{1}{2}\int_{\R^{N}} V(\e_{n}z+y_{n}) \tilde{v}_{n}^{2} \, dx-\Sigma_{0}(\tilde{v}_{n})  \right] \\
&\leq \liminf_{n\rightarrow \infty} J_{\e_{n}}(t_{n} u_{n}) \leq \liminf_{n\rightarrow \infty} J_{\e_{n}}(u_{n})=c_{V_{0}}
\end{align*}
which gives a contradiction.

\end{proof}

\noindent
Now, we introduce a subset $\tilde{\mathcal{N}}_{\e}$ of $\mathcal{N}_{\e}$ by setting 
$$
\tilde{\mathcal{N}}_{\e}=\{u\in \mathcal{N}_{\e}: J_{\e}(u)\leq c_{V_{0}}+h(\e)\},
$$
where $h:\R_{+}\rightarrow \R_{+}$ is such that $h(\e)\rightarrow 0$ as $\e \rightarrow 0$.
Given $y\in M$, we can use Lemma \ref{lemma3.4} to conclude that $h(\e)=|J_{\e}(\Phi_{\e}(y))-c_{V_{0}}|\rightarrow 0$ as $\e \rightarrow 0$. Hence $\Phi_{\e}(y)\in \tilde{\mathcal{N}}_{\e}$, and $\tilde{\mathcal{N}}_{\e}\neq \emptyset$ for any $\e>0$. Moreover, we have the following lemma.

\begin{lem}\label{lemma3.7}
$$
\lim_{\e \rightarrow 0} \sup_{u\in \tilde{\mathcal{N}}_{\e}} dist(\beta_{\e}(u), M_{\delta})=0.
$$
\end{lem}
\begin{proof}
Let $\e_{n}\rightarrow 0$ as $n\rightarrow \infty$. For any $n\in \N$, there exists $u_{n}\in \tilde{\mathcal{N}}_{\e_{n}}$ such that
$$
\sup_{u\in \tilde{\mathcal{N}}_{\e_{n}}} \inf_{y\in M_{\delta}}|\beta_{\e_{n}}(u)-y|=\inf_{y\in M_{\delta}}|\beta_{\e_{n}}(u_{n})-y|+o_{n}(1).
$$
Therefore, it is suffices to prove that there exists $(y_{n})\subset M_{\delta}$ such that 
\begin{equation}\label{3.13}
\lim_{n\rightarrow \infty} |\beta_{\e_{n}}(u_{n})-y_{n}|=0.
\end{equation}
We note that $(u_{n})\subset  \tilde{\mathcal{N}}_{\e_{n}}\subset  \mathcal{N}_{\e_{n}}$, from which we deuce that
$$
c_{V_{0}}\leq c_{\e_{n}}\leq J_{\e_{n}}(u_{n})\leq c_{V_{0}}+h(\e_{n}).
$$
This yields $J_{\e_{n}}(u_{n})\rightarrow c_{V_{0}}$. By using Lemma \ref{prop3.3}, there exists $(\tilde{y}_{n})\subset \R^{N}$ such that $y_{n}=\e_{n}\tilde{y}_{n}\in M_{\delta}$ for $n$ sufficiently large. By setting $v_{n}=u_{n}(\cdot+\tilde{y}_{n})$ and using a change of variable, we can see that
$$
\beta_{\e_{n}}(u_{n})=y_{n}+\frac{\int_{\R^{N}}[\Upsilon(\e_{n}z+y_{n})-y_{n}] v_{n}^{2}(z) \, dz}{\int_{\R^{N}} v^{2}_{n}(z)\, dz}.
$$
Since $\e_{n} z+y_{n}\rightarrow y\in M$, we deduce that $\beta_{\e_{n}}(u_{n})=y_{n}+o_{n}(1)$, that is (\ref{3.13}) holds.

\end{proof}

\section{Proof of Theorem \ref{thmf}}
This last section is devoted to the proof of the main result of this work. Firstly, we show that \eqref{Pe} admits at least $cat_{M_{\delta}}(M)$ positive solutions.
In order to achieve our aim, we recall the following result for critical points involving Ljusternik-Schnirelmann category. For the details of the proof one can see \cite{MW}.
\begin{thm}\label{LSt}
Let $U$ be a $C^{1,1}$ complete Riemannian manifold (modelled on a Hilbert space). Assume that $h\in C^{1}(U, \R)$ bounded from below and satisfies $-\infty<\inf_{U} h<d<k<\infty$. Moreover, suppose that $h$ satisfies Palais-Smale condition on the sublevel $\{u\in U: h(u)\leq k\}$ and that $d$ is not a critical level for $h$. Then
$$ 
card\{u\in h^{d}: \nabla h(u)=0\}\geq cat_{h^{d}}(h^{d}).
$$
\end{thm}
\noindent
Since $\mathcal{N}_{\e}$ is not a $C^{1}$ submanifold of $H^{s}_{\e}$, we cannot apply Theorem \ref{LSt} directly. Fortunately, from Lemma \ref{lemz1}, we know that the mapping $m_{\e}$ is a homeomorphism between $\mathcal{N}_{\e}$ and $\mathbb{S}_{\e}$, and $\mathbb{S}_{\e}$ is a $C^{1}$ submanifold of $H^{s}_{\e}$. So we can apply Theorem \ref{LSt} to 
$\psi_{\e}(u)=J_{\e}(\hat{m}_{\e}(u))|_{\mathbb{S}_{\e}}=J_{\e}(m_{\e}(u))$, where $\psi_{\e}$ is given in Proposition \ref{propz2}.
\begin{thm}\label{teorema}
Assume that $(V_1)$-$(V_2)$ and $(f_1)$-$(f_4)$ hold. Then, for any $\delta>0$ there exists $\bar{\e}_\delta>0$ such that, for any $\e \in (0, \bar{\e}_\delta)$, problem $(\ref{Pe})$ has at least $cat_{M_{\delta}}(M)$ positive solutions.   
\end{thm}

\begin{proof}
For any $\e>0$, we define $\alpha_\e : M \rightarrow \mathbb{S}_{\e}$ by setting $\alpha_\e(y)= m_{\e}^{-1}(\Phi_{\e}(y))$. By using Lemma \ref{lemma3.4} and the definition of $\psi_{\e}$, we can see that
\begin{equation*}
\lim_{\e \rightarrow 0} \psi_{\e}(\alpha_{\e}(y)) = \lim_{\e \rightarrow 0} J_{\e}(\Phi_{\e}(y))= c_{V_{0}} \mbox{ uniformly in } y\in M. 
\end{equation*}  
Then, there exists $\bar{\e}>0$ such that $\tilde{\mathbb{S}}_{\e}:=\{ w\in \mathbb{S}_{\e} : \psi_{\e}(w) \leq c_{V_{0}} + h(\e)\} \neq \emptyset$ for all $\e \in (0, \bar{\e})$. \\
Taking into account Lemma \ref{lemma3.4}, Lemma \ref{lemz1}-(c), Lemma \ref{lemma3.5} and Lemma \ref{lemma3.7}, we can find $\bar{\e}= \bar{\e}_{\delta}>0$ such that the following diagram
\begin{equation*}
M\stackrel{\Phi_{\e}}{\rightarrow} \tilde{\mathcal{N}}_{\e} \stackrel{m_{\e}^{-1}}{\rightarrow} \tilde{\mathbb{S}}_{\e} \stackrel{m_{\e}}{\rightarrow} \tilde{\mathcal{N}}_{\e} \stackrel{\beta_{\e}}{\rightarrow} M_{\delta}
\end{equation*}    
is well defined for any $\e \in (0, \bar{\e})$. \\
By using Lemma \ref{lemma3.5}, there exists a function $\theta(\e, y)$ with $|\theta(\e, y)|<\frac{\delta}{2}$ uniformly in $y\in M$ for all $\e \in (0, \bar{\e})$ such that $\beta_{\e}(\Phi_{\e}(y))= y+ \theta(\e, y)$ for all $y\in M$. Then, we can see that $H(t, y)= y+ (1-t)\theta(\e, y)$ with $(t, y)\in [0,1]\times M$ is a homotopy between $\beta_{\e} \circ \Phi_{\e}=(\beta_{\e}\circ m_{\e}) \circ \alpha_{\e}$ and the inclusion map $id: M \rightarrow M_{\delta}$. This fact and Lemma $4.3$ in \cite{BC} implies that $cat_{\tilde{\mathbb{S}}_{\e}} (\tilde{\mathbb{S}}_{\e})\geq cat_{M_{\delta}}(M)$. On the other hand, let us choose a function $h(\e)>0$ such that $h(\e)\rightarrow 0$ as $\e\rightarrow 0$ and such that $c_{V_{0}}+h(\e)$ is not a critical level for $J_{\e}$. For $\e>0$ small enough, we deduce from Lemma \ref{lemma2.10} that $\psi_{\e}$ satisfies the Palais-Smale condition in $\tilde{\mathbb{S}}_{\e}$. So, it follows from Theorem \ref{LSt} that $\psi_{\e}$ has at least $cat_{\tilde{\mathbb{S}}_{\e}}(\tilde{\mathbb{S}}_{\e})$ critical points on $\tilde{\mathbb{S}}_{\e}$. By Proposition \ref{propz2}-(d) we conclude that $J_{\e}$ admits at least $cat_{M_{\delta}}(M)$ critical points. 
\end{proof}
\noindent

Now, we use a Moser iteration argument \cite{Moser} which will be fundamental to study of behavior of the maximum points of the solutions.
\begin{lem}\label{lemMoser}
Let $\e_{n}\rightarrow 0$ and $u_{n}\in \widetilde{\mathcal{N}}_{\e_{n}}$ be a solution to \eqref{Pe}. Then $v_{n}=u_{n}(\cdot+\tilde{y}_{n})$ satisfies the following problem
\begin{equation}\label{Pn}
\left\{
\begin{array}{ll}
(-\Delta)^{s}v_{n} + V_{n}(x) v_{n}= \left(\frac{1}{|x|^{\mu}}*G_{n}(v_{n})\right)g_{n}(v_{n}) &\mbox{ in } \R^{N}\\
v_{n}\in H^{s}(\R^{N})\\
v_{n}>0 &\mbox{ in } \R^{N}, 
\end{array}
\right.
\end{equation}
where $V_{n}(x)= V(\e_{n}x+ \e_{n}\tilde{y}_{n})$, $g_{n}(v_{n})=g(\e_{n} x+\e_{n}\tilde{y}_{n}, v_{n})$, $\e_{n}\tilde{y}_{n}\rightarrow y\in M$, and 
there exists $C>0$ such that 
$\|v_{n}\|_{L^{\infty}(\R^{N})}\leq C$ for all $n\in \mathbb{N}$. 

\end{lem}
\begin{proof}
For any $L>0$ and $\beta>1$, let us define the function 
\begin{equation*}
\gamma(v_{n})=\gamma_{L, \beta}(v_{n})=v_{n} v_{L, n}^{2(\beta-1)}\in \h
\end{equation*}
where  $v_{L,n}=\min\{u_{n}, L\}$. 
Since $\gamma$ is an increasing function, we have
\begin{align*}
(a-b)(\gamma(a)- \gamma(b))\geq 0 \quad \mbox{ for any } a, b\in \R.
\end{align*}
Let us consider 
\begin{equation*}
\mathcal{E}(t)=\frac{|t|^{2}}{2} \quad \mbox{ and } \quad \Gamma(t)=\int_{0}^{t} (\gamma'(\tau))^{\frac{1}{2}} d\tau. 
\end{equation*}
Then, by applying Jensen inequality we get for all $a, b\in \R$ such that $a>b$,
\begin{align*}
\mathcal{E}'(a-b)(\gamma(a)-\gamma(b)) &=(a-b) (\gamma(a)-\gamma(b))= (a-b) \int_{a}^{b} \gamma'(t) dt \\
&= (a-b) \int_{a}^{b} (\Gamma'(t))^{2} dt \geq \left(\int_{a}^{b} (\Gamma'(t)) dt\right)^{2}.
\end{align*}
The same argument works when $a\leq b$. Therefore
\begin{equation}\label{Gg}
\mathcal{E}'(a-b)(\gamma(a)-\gamma(b))\geq |\Gamma(a)-\Gamma(b)|^{2} \mbox{ for any } a, b\in\R. 
\end{equation}
By using \eqref{Gg}, we can see that
\begin{align}\label{Gg1}
|\Gamma(v_{n})(x)- \Gamma(v_{n})(y)|^{2} \leq (v_{n}(x)- v_{n}(y))((v_{n}v_{L,n}^{2(\beta-1)})(x)- (v_{n}v_{L,n}^{2(\beta-1)})(y)). 
\end{align}
Choosing $\gamma(v_{n})=v_{n} v_{L, n}^{2(\beta-1)}$ as test-function in \eqref{Pn}, and using \eqref{Gg1}, we obtain
\begin{align}\label{BMS}
&[\Gamma(v_{n})]^{2}+\int_{\R^{N}} V_{n}(x)|v_{n}|^{2}v_{L, n}^{2(\beta-1)} dx \nonumber \\
&\leq \iint_{\R^{2N}} \frac{(v_{n}(x)- v_{n}(y))}{|x-y|^{N+2s}} ((v_{n}v_{L, n}^{2(\beta-1)})(x)-(v_{n} v_{L,n}^{2(\beta-1)})(y)) \,dx dy +\int_{\R^{N}} V_{n}(x)|v_{n}|^{2}v_{L,n}^{2(\beta-1)} dx \nonumber\\
&=\int_{\R^{N}} \left(\frac{1}{|x|^{\mu}}*G_{n}(v_{n}) \right)g_{n}(v_{n}) v_{n} v_{L,n}^{2(\beta-1)} dx.
\end{align}
Since 
$$
\Gamma(v_{n})\geq \frac{1}{\beta} v_{n} v_{L,n}^{\beta-1},
$$ 
and using Theorem \ref{Sembedding}, we have
\begin{equation}\label{SS1}
[\Gamma(v_{n})]^{2}\geq S_{*} \|\Gamma(v_{n})\|^{2}_{L^{\p}(\R^{N})}\geq \left(\frac{1}{\beta}\right)^{2} S_{*}\|v_{n} v_{L,n}^{\beta-1}\|^{2}_{L^{\p}(\R^{N})}.
\end{equation}
On the other hand, from the boundedness of $(v_{n})$, it follows that there exists $C_{0}>0$ such that
\begin{equation}\label{GbAY}
\sup_{n\in \mathbb{N}} \left\|\frac{1}{|x|^{\mu}}*G(\e x, v_{n}) \right\|_{L^{\infty}(\R^{N})}<C_{0}.
\end{equation}
By the assumption $(g_1)$ and $(g_2)$, for any $\xi>0$ there exists $C_{\xi}>0$ such that
\begin{equation}\label{SS2}
|g_{n}(v_{n})|\leq \xi |v_{n}|+C_{\xi}|v_{n}|^{q-1}.
\end{equation}
Taking $\xi\in (0, V_{0})$, and using \eqref{SS1}, \eqref{GbAY} and \eqref{SS2}, we can see that \eqref{BMS} yields
\begin{align*}
\|v_{n} v_{L,n}^{\beta-1}\|^{2}_{L^{\p}(\R^{N})}&\leq C\beta^{2} \int_{\R^{N}} |v_{n}|^{q} v_{L,n}^{2(\beta-1)} dx. 
\end{align*}
Set $w_{L,n}:=v_{n} v_{L,n}^{\beta-1}$. By applying H\"older inequality, we get
\begin{align*}
\|w_{L,n}\|_{L^{\p}(\R^{N})}^{2}\leq C \beta^{2}  \left(\int_{\R^{N}} v_{n}^{\p}\, dx\right)^{\frac{q-2}{\p}} \left(\int_{\R^{N}} w_{L,n}^{\alpha^{*}_{s}} \, dx\right)^{\frac{2}{\alpha^{*}_{s}}}
\end{align*}
where $\alpha^{*}_{s}:=\frac{2 \p}{\p-(q-2)} \in (2, \p)$. \\
Since $(v_{n})$ is bounded in $\h$, we deduce that
\begin{align}\label{conto4}
\|w_{L,n}\|_{L^{\p}(\R^{N})}^{2}\leq C \beta^{2} \|w_{L,n}\|_{L^{\alpha^{*}_{s}}(\R^{N})}^{2}. 
\end{align}
Now, we observe that if $v_{n}^{\beta}\in L^{\alpha^{*}_{s}}(\R^{N})$, from the definition of $w_{L,n}$, and by using the fact that $v_{L,n}\leq v_{n}$ and  (\ref{conto4}), we obtain
\begin{align}\label{conto5}
\|w_{L,n}\|_{L^{\p}(\R^{N})}^{2}\leq C \beta^{2}  \left(\int_{\R^{N}} v_{n}^{\beta \alpha^{*}_{s}}\, dx\right)^{\frac{2}{\alpha^{*}_{s}}}<\infty.
\end{align}
By passing to the limit in (\ref{conto5}) as $L \rightarrow +\infty$, the Fatou's Lemma yields
\begin{align}\label{conto6}
\|v_{n}\|_{L^{\beta \p}(\R^{N})}\leq C^{\frac{1}{\beta}} \beta^{\frac{1}{\beta}} \|v_{n}\|_{L^{\beta \alpha^{*}_{s}}(\R^{N})}
\end{align}
whenever $v_{n}^{\beta \alpha^{*}_{s}}\in L^{1}(\R^{N})$. 

Now, we set $\beta:=\frac{2^{*}_{s}}{\alpha^{*}_{s}}>1$, and we observe that, being $v_{n}\in L^{2^{*}_{s}}(\R^{N})$, the above inequality holds for this choice of $\beta$. Then, observing that $\beta^{2}\alpha^{*}_{s}=\beta 2^{*}_{s}$, it follows that \eqref{conto6} holds with $\beta$ replaced by $\beta^{2}$.
Therefore, we can see that
\begin{align*}
\|v_{n}\|_{L^{\beta^{2} 2^{*}_{s}}(\R^{N})}\leq C^{\frac{1}{\beta^{2}}} \beta^{\frac{2}{\beta^{2}}} \|v_{n}\|_{L^{\beta^{2} \alpha^{*}_{s}}(\R^{N})}\leq  C^{\left(\frac{1}{\beta}+\frac{1}{\beta^{2}}\right)} \beta^{\frac{1}{\beta}+\frac{2}{\beta^{2}}} \|v_{n}\|_{L^{\beta \alpha^{*}_{s}}(\R^{N})}.
\end{align*}
Iterating this process, and recalling that $\beta \alpha^{*}:=2^{*}_{s}$, we can infer that for every $m\in \mathbb{N}$
\begin{align}\label{conto7}
\|v_{n}\|_{L^{\beta^{m} 2^{*}_{s}}(\R^{N})}\leq C^{\sum_{j=1}^{m}\frac{1}{\beta^{j}}} \beta^{\sum_{j=1}^{m} j\beta^{-j}} \|v_{n}\|_{L^{2^{*}_{s}}(\R^{N})}.
\end{align}
Taking the limit in (\ref{conto7}) as $m \rightarrow +\infty$ and recalling that $\|v_{n}\|_{L^{\p}(\R^{N})}\leq K$, we get
\begin{align*}
\|v_{n}\|_{L^{\infty}(\R^{N})}\leq C^{\sigma_{1}} \beta^{\sigma_{2}}K,
\end{align*}
where 
$$
\sigma_{1}:=\sum_{j=1}^{\infty}\frac{1}{\beta^{j}}<\infty \quad \mbox{ and } \quad \sigma_{2}:=\sum_{j=1}^{\infty}\frac{j}{\beta^{j}}<\infty.
$$
\end{proof}

\noindent
At this point, we are ready to give the proof of our main result.
\begin{proof}[Proof of Theorem \ref{thmf}]
Take $\delta>0$ such that $M_\delta \subset \Lambda$. We begin proving that there exists $\tilde{\e}_{\delta}>0$ such that for any $\e \in (0, \tilde{\e}_{\delta})$ and any solution $u_{\e} \in \tilde{\mathcal{N}}_{\e}$ of \eqref{Pe}, it holds 
\begin{equation}\label{infty}
\|u_{\e}\|_{L^{\infty}(\R^{N}\setminus \Lambda_{\e})}<a. 
\end{equation}
Assume by contradiction that there exist $\e_{n}\rightarrow 0$, $u_{\e_{n}}\in \tilde{\mathcal{N}}_{\e_{n}}$ such that $J'_{\e_{n}}(u_{\e_{n}})=0$ and $\|u_{\e_{n}}\|_{L^{\infty}(\R^{N}\setminus \Lambda_{\e_{n}})}\geq a$. 
Since $J_{\e_{n}}(u_{\e_{n}}) \leq c_{V_{0}} + h(\e_{n})$ and $h(\e_{n})\rightarrow 0$, we can argue as in the  first part of the proof of Lemma \ref{prop3.3}, to deduce that $J_{\e_{n}}(u_{\e_{n}})\rightarrow c_{V_{0}}$.
Then, by using Lemma \ref{prop3.3}, we can find $(\tilde{y}_{n})\subset \R^{N}$ such that $y_{n}:=\e_{n}\tilde{y}_{n}\rightarrow y_{0} \in M$. \\
Now, if we choose $r>0$ such that $B_{r}(y_{0})\subset B_{2r}(y_{0})\subset \Lambda$, we can see $B_{\frac{r}{\e_{n}}}(\frac{y_{0}}{\e_{n}})\subset \Lambda_{\e_{n}}$. In particular, for any $y\in B_{\frac{r}{\e_{n}}}(\tilde{y}_{n})$ there holds
\begin{equation*}
\left|y - \frac{y_{0}}{\e_{n}}\right| \leq |y- \tilde{y}_{n}|+ \left|\tilde{y}_{n} - \frac{y_{0}}{\e_{n}}\right|<\frac{2r}{\e_{n}}\, \mbox{ for } n \mbox{ sufficiently large. }
\end{equation*}
Therefore 
$$
\R^{N}\setminus \Lambda_{\e_{n}}\subset \R^{N} \setminus B_{\frac{r}{\e_{n}}}(\tilde{y}_{n}) \mbox{ for any } n \mbox{ big enough. }
$$ 
Now, denoting by $v_{n}(x)=u_{\e_{n}}(x+ \tilde{y}_{n})$, we can see that 
$$
(-\Delta)^{s}v_{n}+v_{n}=h_{n} \mbox{ in } \R^{N},
$$
where
$$
h_{n}:=v_{n}-V_{n}(x)v_{n}+\left(\frac{1}{|x|^{\mu}}*G_{n}(v_{n})\right)g_{n}(v_{n}),
$$
and $v_{n}\rightarrow v$ converges strongly in $L^{p}(\R^{N})$ for any $p\in [2, \infty)$, in view of Lemma \ref{lemMoser}.

Since $\langle J'_{\e_{n}}(u_{\e_{n}}), u_{\e_{n}}\rangle=0$ and $J_{\e_{n}}(u_{\e_{n}})\rightarrow c_{V_{0}}$, we may assume $u_{\e_{n}}\in \mathcal{B}$ for all $n$ big enough, so that
\begin{equation*}
\left\|\frac{1}{|x|^{\mu}}*G(\e x, u_{\e_{n}})\right\|_{L^{\infty}(\R^{N})}<C_{0}.
\end{equation*}
As a consequence, recalling that $\e_{n}\tilde{y}_{n}\rightarrow y_{0}\in M$, we get
$$
\|h_{n}\|_{L^{\infty}(\R^{N})}\leq C\quad \mbox{ and } \quad h_{n}\rightarrow v-V(y_{0})v+\left(\frac{1}{|x|^{\mu}}*F(v)\right)f(v) \mbox{ in } L^{p}(\R^{N}) \quad \forall p\in [2, \infty).
$$
Hence, $v_{n}=\mathcal{K}*h_{n}$, where $\mathcal{K}$ is the Bessel kernel \cite{FQT}, and we can argue as in \cite{AM} to prove that 
$$
\lim_{|x|\rightarrow \infty} \sup_{n\in \mathbb{N}} |v_{n}(x)|=0,
$$
which implies that there exists $R>0$ such that $v_{n}(x)<a$ for $|x|\geq R$ and $n\in \N$. \\
Thus 
$$
u_{\e_{n}}(x)<a \mbox{ for any } x\in \R^{N}\setminus B_{R}(\tilde{y}_{n}), \, n\in \N.
$$ 
As a consequence, there exists $\nu \in \N$ such that for any $n\geq \nu$ and $\frac{r}{\e_{n}}>R$, it holds 
$$
\R^{N}\setminus \Lambda_{\e_{n}}\subset \R^{N} \setminus B_{\frac{r}{\e_{n}}}(\tilde{y}_{n})\subset \R^{N}\setminus B_{R}(\tilde{y}_{n}),
$$
which gives $u_{\e_{n}}(x)<a$ for any $x\in \R^{N}\setminus \Lambda_{\e_{n}}$, and this is impossible. \\
Now, let $\bar{\e}_{\delta}$ given in Theorem \ref{teorema} and take $\e_{\delta}= \min \{\tilde{\e}_{\delta}, \bar{\e}_{\delta}\}$. Fix $\e \in (0, \e_{\delta})$. By Theorem \ref{teorema}, we know that problem \eqref{Pe} admits $cat_{M_{\delta}}(M)$ nontrivial solutions $u_{\e}$. Since $u_{\e}\in \tilde{\mathcal{N}}_{\e}$ satisfies \eqref{infty}, from the definition of $g$ it follows that $u_{\e}$ is a solution of \eqref{R}. \\
Now, we study the behavior of the maximum points of $u_{n}\in H^{s}_{\e_{n}}$ solutions to the problem \eqref{Pe}.
Let us observe that $(g_1)$ implies that we can find $\gamma\in (0, a)$ such that 
\begin{equation}\label{4.4FS}
g(\e x, t)t\leq \frac{V_{0}}{\ell_{0}} t^{2} \mbox{ for any } x\in \R^{N}, t\leq \gamma.
\end{equation}
Arguing as before, we can find $R>0$ such that 
\begin{equation}\label{4.5FS}
\|u_{n}\|_{L^{\infty}(B^{c}_{R}(\tilde{y}_{n}))}<\gamma.
\end{equation}
Moreover, up to extract a subsequence, we may assume that 
\begin{equation}\label{4.6FS}
\|u_{n}\|_{L^{\infty}(B_{R}(\tilde{y}_{n}))}\geq \gamma.
\end{equation}
Indeed, if \eqref{4.6FS} does not hold, in view of \eqref{4.5FS} we can see that $\|u_{n}\|_{L^{\infty}(\R^{N})}<\gamma$. Then, by using $\langle J'_{\e_{n}}(u_{n}), u_{n}\rangle=0$, \eqref{4.4FS} and  
$$
\left\|\frac{1}{|x|^{\mu}}*G(\e x, u_{n})\right\|_{L^{\infty}(\R^{N})}<C_{0},
$$
we can infer
\begin{equation*}
\|u_{n}\|_{\e_{n}}^{2}\leq C_{0}\int_{\R^{N}} g(\e_{n} x, u_{n}) u_{n} \,dx\leq \frac{C_{0}}{\ell_{0}} \int_{\R^{N}} V_{0} u_{n}^{2} \, dx
\end{equation*}
which together with $\frac{C_{0}}{\ell_{0}}<\frac{1}{2}$ yields $\|u_{n}\|_{\e_{n}}=0$, and this gives a contradiction.
As a consequence, \eqref{4.6FS} holds. Taking into account \eqref{4.5FS} and \eqref{4.6FS} we can deduce that the maximum points $p_{n}\in \R^{N}$ of $u_{n}$ belong to $B_{R}(\tilde{y}_{n})$. Therefore, $p_{n}=\tilde{y}_{n}+q_{n}$ for some $q_{n}\in B_{R}(0)$. Hence, $\eta_{n}=\e_{n} \tilde{y}_{n}+\e_{n} q_{n}$ is the maximum point of $\hat{u}_{n}(x)=u_{n}(x/\e_{n})$. Since $|q_{n}|<R$ for any  $n\in \mathbb{N}$ and $\e_{n} \tilde{y}_{n}\rightarrow y_{0}\in M$, from the continuity of $V$ we can infer that
$$
\lim_{n\rightarrow \infty} V(\eta_{\e_{n}})=V(y_{0})=V_{0},
$$
which ends the proof of the Theorem \ref{thmf}.
\end{proof}

\noindent {\bf Acknowledgements.} 
The author would like to express his sincere thanks to the anonymous referee for his/her careful reading of the manuscript and valuable comments and suggestions.\\
The paper has been carried out under the auspices of the INdAM - GNAMPA Project 2017 titled: {\it Teoria e modelli per problemi non locali}.

\end{document}